\documentclass{amsart}
\usepackage{subfiles}
\usepackage[colorlinks]{hyperref}

\usepackage{enumerate}
\usepackage{amssymb,euscript,amsmath, mathrsfs, tensor}
\usepackage{stmaryrd}
\usepackage[dvips]{graphicx}
\usepackage[dvips]{color}
\usepackage[all,cmtip]{xy}
\usepackage{xcolor}
\usepackage[margin=1.2in]{geometry}
\usepackage{tikz}
\usetikzlibrary{cd}

\newcounter{ENUM}

\def\To#1{\buildrel\hbox{\tiny{$#1$}}\over\longrightarrow}

 \title{Finiteness for Hecke algebras of $p$-adic groups}
\author{Jean-Fran\c cois Dat}
\address{Jean-Fran\c cois Dat, Institut de Math\'ematiques de Juss\'ieu, 4 Place Jussieu, 75252, Paris.}
\email{jean-francois.dat@imj-prg.fr }
\author{David Helm}
\address{David Helm, Department of Mathematics, Imperial College, London, SW7 2AZ, United Kingdom.}
\email{d.helm@imperial.ac.uk }
\author{Robert Kurinczuk}
\address{Robert Kurinczuk, School of Mathematics and Statistics, University of Sheffield, Sheffield, S3 7RH, United Kingdom.}
\email{robkurinczuk@gmail.com}
\author{Gilbert Moss}
\address{Gil Moss, Department of Mathematics, The University of Utah, Salt Lake City, UT 84112, USA.}
\email{moss@math.utah.edu}

\def\GG{\mathbb{G}}
\def\QQ{\mathbb{Q}}
\def\ZZ{\mathbb{Z}}
\def\ZM{\mathbb{Z}}
\def\NM{\mathbb{N}}
\def\FF{\mathbb{F}}

\def\Spec{\mathrm{Spec}}

\def\Hom{\mathrm{Hom}}
\def\End{\mathrm{End}}

\def\Fr{\mathrm{Fr}}
\def\Rep{\mathrm{Rep}}
\def\ind{\mathrm{ind}}

\def\Irr{\mathrm{Irr}}

\def\Exc{\mathrm{Exc}}

\def\LG{{\tensor*[^L]\G{}}}

\def\ZZz{\mathfrak{Z}}
\def\ZG{\mathfrak{Z}}
\def\HG{{\hat G}}
\def\HM{{\hat M}}
\def\HB{{\hat B}}
\def\HT{{\hat T}}
\def\HH{{\hat H}}

\def\Ql{\overline{\mathbb{Q}}_\ell}

\def\I{\mathrm{I}}

\def\G{\mathrm{G}}

\def\I{\mathrm{I}}

\def\R{\mathrm{R}}

\def\Irr{\mathrm{Irr}}

\def\OC{{\mathcal O}}
\def\EC{{\mathfrak E}}

\newcommand{\margh}[1]{}
\theoremstyle{theorem}

\newtheorem{theorem}{Theorem}[section]

\newtheorem{lemma}[theorem]{Lemma}
\newtheorem{corollary}[theorem]{Corollary}

\newtheorem{remark}[theorem]{Remark}
\numberwithin{equation}{section}

\def\Zl{\overline{\mathbb{Z}_\ell}}

\def\GL{\mathrm{GL}}

\def\Zlb{\mathbb{Z}_\ell}
\def\Qlb{\mathbb{Q}_\ell}

\begin{document}

\maketitle

\begin{abstract}
  Let $G$ be a reductive group over a non-archimedean local field $F$ of residue
  characteristic $p$. We prove that the Hecke algebras of $G(F)$ with coefficients in any
  noetherian $\Zlb$-algebra $R$
  with $\ell\neq p$, are finitely generated modules over their centers, and that these centers
  are finitely generated $R$-algebras.  Following Bernstein's original strategy, we then
  deduce that ``second adjointness'' holds for smooth representations of
  $G(F)$ with coefficients in any $\ZZ[\frac{1}{p}]$-algebra.  These results had been conjectured for a long time. The crucial new tool that unlocks the problem  is the Fargues-Scholze
  morphism between a certain ``excursion algebra'' defined on the 
  Langlands parameters side and the Bernstein center of $G(F)$. Using this bridge, our
  main results are representation theoretic counterparts of the finiteness of certain
  morphisms between coarse moduli spaces of local Langlands parameters that we also prove here,
  which may be of independent interest.
  \end{abstract}

\tableofcontents
  
\section{Main results}

  Let $\GG$ be a reductive group over a non-archimedean local field $F$ of residue
  characteristic $p$. The group
  $G:=\GG(F)$ is then locally profinite, hence for any
  open compact subgroup $H$ of $G$, the free abelian group
  $\ZZ[H\backslash G/H]$ carries a structure of an associative ring,  called a Hecke
  ring. One of the main results of this paper is the following statement.
  \begin{theorem}\label{thm_finiteness}
    For any prime $\ell\neq p$ and any noetherian $\ZZ_{\ell}$-algebra $R$, the base
    change
    $R[H\backslash G/H]$ is
    a finitely generated module over 
    its center, which is a finitely generated (commutative) $R$-algebra.
  \end{theorem}
  %Note that this statement implies a similar statement after base change to any noetherian
  %$\ZZ_{\ell}$-algebra $R$. For such a coefficient ring $R$,
  Here is an equivalent formulation in terms of smooth representations. Denote by
  ${\rm  Rep}_{R}(G)$ the category of all smooth $R G$-modules and by
  $\ZZz_{R}(G)$ the center of this category. We define an $RG$-module to be
  \emph{$\ZZz$-finite} if
  \begin{itemize}
  \item the image of $\ZZz_{R}(G)\To{}{\rm End}_{RG}(V)$ is a finitely
    generated $R$-algebra, and
  \item $V$ is admissible over $\ZZz_{R}(G)$, i.e. $V^{H}$ is a finitely generated
    $\ZZz_{R}(G)$-module for any compact open subgroup $H$ of $G$.
  \end{itemize}
  % Then, with the same hypothesis on $R$, the
  Then the
  above theorem is equivalent to
  the following one (see Lemma \ref{lemma_equivalence_main_thms}).
  \begin{theorem}\label{thm_Z-admissible}
    (Same hypothesis on $R$.) Any finitely generated smooth $R G$-module $V$ is $\ZZz$-finite.
    % admissible over
    % $\ZZz_{R}(G)$, and the image of $\ZZz_{R}(G)\To{}{\rm End}_{RG}(V)$ is a finitely
    % generated $R$-algebra.
  \end{theorem}
  
  When $R =\overline\QQ_{\ell}$, 
  % The same statements with coefficients $\overline\QQ_{\ell}$ instead of $\ZZ_{\ell}$
  these statements are
  famous theorems of Bernstein. For  $R=\ZZ_{\ell}$ or $R=\FF_{\ell}$, the only
  previously known case was for $G={\rm GL}_{n}$ in \cite{HelmBC}, and the proof there
  relies on very specific features of ${\rm GL}_{n}$ such as Vigneras' ``uniqueness of
  supercuspidal support''.   There have been also
  partial results for more general groups. For example, in
   \cite{datfinitude}, the fact that $R[H\backslash G/H]$ is a
   noetherian ring (albeit with no control on the center)  was proved to be
   implied by the so-called ``second adjointness'' between parabolic
  functors, and the latter property was established  for groups having a suitable form of type theory,  such as classical groups or ``very tame'' groups.  However, this second adjointness
  property was first discovered by
  Bernstein for complex representations as a consequence of his theorem on finiteness of
  Hecke algebras. Following his argument,
  %and assuming that $R$ contains a   square root of $q$ in order to
  % normalize parabolic functors,
  we also prove :

  \begin{corollary}[Second adjointness] \label{cor_second_adj}
    For any $\ZZ[\frac{1}{p}]$-algebra $R$, and
    for all pairs of opposite parabolic subgroups $(P,\bar P)$ in $G$ with common Levi
    component $M=P\cap \bar P$, the twisted opposite Jacquet functor
    % $r_{\bar P}:\Rep_{\ZM_{\ell}[\sqrt q]}(G)\To{}\Rep_{\ZM_{\ell}[\sqrt q]}(M)$
     $\delta_{P}. \R_{\bar P}:\Rep_{R}(G)\To{}\Rep_{R}(M)$
    is right adjoint to the  parabolic induction functor
    % $i_{P}:\Rep_{\ZM_{\ell}[\sqrt q]}(M)\To{}\Rep_{\ZM_{\ell}[\sqrt q]}(G)$.
    $\I_{P}:\Rep_{R}(M)\To{}\Rep_{R}(G)$, where $\delta_{P}: M\To{}
      \ZZ[\frac{1}{p}]^{\times}$ denotes the modulus character of $P$.
  \end{corollary}

Note that when one can fix a square root of $q$ in $R$, the 
 parabolic functors can be normalized by putting $i_{P}:=\I_{P} \delta_P^{\frac{1}{2}}$ and
 $r_{P}:=\delta_{P}^{-\frac 12}.\R_{P}$, and the twist by the modulus character disappears from the 
 statement of Corollary \ref{cor_second_adj}, as in Bernstein's original result. 
 
 It is perhaps surprising that we can deduce second adjointness over $\ZZ[\frac{1}{p}]$ from finiteness results that we can only establish over $\Zlb$ for $\ell \neq p$.  The key point is to deduce a certain ``stability'' property for objects of $\Rep_{\ZZ[\frac{1}{p}]}(G)$; we reduce this problem to establishing ``stability'' for certain injective objects $I_{\ell}$ of $\Rep_{\ZZ[\frac{1}{p}]}(G)$ that naturally have the additional structure of a $\Zlb$-module, allowing us to apply the finiteness results we have proven over $\Zlb$.
 
Let us quote three further consequences, for which only partial results have been known
so far.

\begin{corollary} \label{cor_noetherian}
For any noetherian $\ZZ[\frac{1}{p}]$-algebra $R$, and any compact open subgroup $H$ of $G$, the Hecke algebra $R[H\backslash G / H]$ is noetherian.
\end{corollary}

Our arguments in this paper fall short of establishing that the rings $R[H \backslash G
/H]$ are finitely generated over their centers when $R$ is not a $\Zlb$-algebra.  We will
address this question in forthcoming work.  Indeed, we expect to be able to prove this,
for $R$ an arbitrary noetherian $\ZZ[\frac{1}{p}]$-algebra $R$, by first proving it over
$\ZZ[\frac 1{pN}]$ for some integer $N$, and then applying flat descent (see Lemma
\ref{lemma_reduc_Zl} (2)) to the map $\ZM[\frac 1p]\To{} \ZM[\frac 1{pN}]\times \prod_{\ell|N}\ZM_{\ell}$.
Here $N$ will be the l.c.m. of the orders of torsion elements in $G$.
In this ``banal'' setting, with the help  of  Corollary~\ref{cor_second_adj}, one can construct
quite explicit projective objects of $\Rep_R(G)$ and thus one has good control over the
center $\ZZz_R(G).$ 

\begin{corollary} \label{cor_induction_restriction} Let $P$ be a parabolic subgroup  with Levi component $M$, and let $R$  be a noetherian $\ZZ[\frac{1}{p}]$-algebra.
  \begin{enumerate}
  \item The parabolic induction functor $\I_{P}$ takes projective, resp. finitely generated, smooth $RM$-modules to
    projective, resp. finitely generated, smooth $RG$-modules.
  \item The Jacquet functor $\R_{P}$ takes admissible $RG$-modules to admissible $RM$-modules.
  \end{enumerate}
\end{corollary}

In turn, item (2) above is the main ingredient to prove the next corollary, which was not
known for general groups even when $\ell$ is a banal prime (that is, when $\ell$ is a prime not dividing the pro-order of any compact open subgroup of $G$).

\begin{corollary}\label{cor_integral}
An irreducible $\overline\QQ_{\ell} G$-representation is
  \emph{integral} (i.e. admits an admissible $G$-stable $\overline\ZM_{\ell}$-lattice) if
  and only if its supercuspidal support is integral.
\end{corollary}

  \medskip

  Let us outline our strategy to prove Theorems \ref{thm_finiteness}
  and  \ref{thm_Z-admissible}. The details are given in Section \ref{sec:finit-group-side}.

  We first observe that it is enough to prove them in the case where $R$ is
  $\ZM_{\ell}$ or any finite flat extension of $\Zlb$. For several
  reasons, it will be convenient to have a fixed square root of $q$ in
  $R$, so we will work over $\ZM_{\ell}':=\ZM_{\ell}[\sqrt q]$ where $\sqrt q\in\Ql$ is
  fixed, and we will use normalized parabolic functors $i_{P}$ and $r_{P}$.
  
  %  Our first tool is the decomposition of $\Rep_{\ZM[\frac 1p]}(G)$ according to depth, in the sense of
  % Moy and Prasad, as in \cite[Appendix]{datfinitude}.
  % Each factor subcategory $\Rep_{\ZM[\frac 1p]}(G)_{r}$ in this decomposition
  % is generated by a finitely generated projective object $P_{r}$.

  Since $\ZG$-finite objects are stable under taking finite direct products and quotients, in order  
  to prove Theorem \ref{thm_Z-admissible}, it is enough to prove 
  that each $\Zlb'G$-module of the form $\Zlb'[G/H]$ for some open pro-$p$-subgroup $H$ of
  $G$ is $\ZG$-finite. For such an $H$, using Bernstein's decomposition over
  $\overline\QQ_{\ell}$, we show that $\Zlb'[G/H]$ can
  be embedded in a $\ZM_{\ell}'G$-module of the form $V_{H}=\bigoplus_{P}i_{P}(W_{P})$,
  where $P=U_{P}M_{P}$ runs through a finite set of parabolic subgroups and $W_{P}$ is a 
  \emph{cuspidal} finitely   generated $\ZM_{\ell}' M_{P}$-module. Since $\ZG$-finite
  objects are stable under taking subobjects, it is enough to prove that
  $V_{H}$ is $\ZZz$-finite.

%  $P_{r}\otimes \Zlb'$ is $\ZZz$-finite in $\Rep_{\Zlb'}(G)$.
  % Then, observe that $P_{r}\otimes\overline\QQ_{\ell}$ is a projective generator of a finite sum of
  % Bernstein blocks. Using Bernstein's theory over $\overline\QQ_{\ell}$, we show that $P_{r}\otimes\Zlb'$ can
  % be embedded in a $\ZM_{\ell}'G$-module of the form $V_{r}=\bigoplus_{P}i_{P}(W_{P,r})$,
  % where $P=U_{P}M_{P}$ runs through a finite set of parabolic subgroups and $W_{P,r}$ is a 
  % \emph{cuspidal} finitely   generated $\ZM_{\ell}' M_{P}$-module. Clearly, if
  % $V_{r}$ is $\ZZz$-finite, then so is $P_{r}\otimes\Zlb'$.
  So we are left with
  proving that any $V$ of the form $i_{P}(W)$ for some finitely generated \emph{cuspidal}
  $\ZM_{\ell}'M$-module $W$, is $\ZZz$-finite
  (even though, at
  this point, it is not even clear whether such a $V$ is finitely generated).  
  Now, by \cite[Lemme 4.2]{datfinitude}, we know that such a $W$ is
  admissible over
  $\ZZz_{\ZM_{\ell}'}(Z_{M})$, where $Z_{M}$ denotes the
  maximal central torus of $M$.

  \medskip

  This is where Fargues and Scholze's local version of V. Lafforgue's theory of
  excursion operators comes in. Denote by $\HG$ the dual
  pinned reductive group scheme of $\GG$ over $\ZM[\frac 1p]$, endowed with its pinning-preserving
  action of the Weil group $W_{F}$ of $F$.
  Recall the subgroup $W_{F}^{0}$ of $W_{F}$ obtained in \cite{DHKM} by ``discretizing tame inertia'',
   and choose a separated decreasing filtration
  $(P_{F}^{e})_{e\in\NM}$ of the wild inertia subgroup $P_{F}\subset W_{F}$ by normal
  subgroups.
  The ``excursion'' $\ZM[\frac 1p]$-algebra
  $$\Exc(W_{F},\HG) :=  {\rm lim}_{e} \Exc(W_{F}^{0}/P_{F}^{e},\HG) \,\hbox{ where }\,
  \Exc(W_{F}^{0}/P_{F}^{e},\HG):= {\rm colim}_{n,F_{n}\rightarrow W_{F}^{0}/P_{F^{e}}} 
  \mathcal{O}(Z^{1}(F_{n},\HG))^{\HG}$$
  can be thought of as the ring of functions on the space of ``$\HG$-valued
  continuous pseudo-characters'' of $W_{F}^{0}$.  Moreover, each
  $\Exc(W_{F}^{0}/P_{F}^{e},\HG)$ is known to be a finitely generated
  (commutative) $\ZM[\frac 1p]$-algebra.
  Using Hecke operators on spaces of $G$-bundles on the Fargues-Fontaine curve, Fargues
  and Scholze have constructed in \cite[Ch. IX]{FarguesScholze}  a map of
  $\Zlb'$-algebras $${\rm FS}_{G}:\, \Exc(W_{F},\hat\G)_{\Zlb'}\longrightarrow\ZZz_{\Zlb'}(G)$$
 that enjoys the following important properties 
  % Then, using Hecke operators on spaces of $G$-bundles on the Fargues-Fontaine curve, Fargues
  % and Scholze have constructed in \cite[Ch. IX]{FarguesScholze}  a map from the excursion
  % algebra to the Bernstein center.
  % However, for technical reasons pertaining to the difference between the L-group and the C-group,
  % it is convenient to work with a fixed choice of square root of $q$, so we put
  % $\ZM_{\ell}':=\ZM_{\ell}[\sqrt q]$ and we denote by
  % $${\rm FS}_{G}:\, \Exc(W_{F},\hat\G)_{\Zlb'}\longrightarrow\ZZz_{\Zlb'}(G)$$
  % the morphism  of
  % $\Zlb'$-algebras %$${\rm FS}_{G}:\, \Exc(W_{F},\hat\G)\longrightarrow\ZZz_{\Zlb}(G)$$
  % constructed by Fargues and Scholze. It enjoys the following important properties
  (details and references are given in section \ref{subsec_proof_main}) :

  \begin{enumerate}
   \item Compatibility with parabolic induction : for any parabolic subgroup $P$ with Levi
     component $M$ and any $\ZM_{\ell}'M$-module $W$, the
     following diagram is commutative
     $$\xymatrix{
       \Exc(W_{F},\HG)_{\Zlb'} \ar[r]^{{\rm FS}_{G}} \ar[d]
       & {\rm End}_{\ZM_{\ell}'G}(i_{P}W)
       \\
       \Exc(W_{F},\HM)_{\Zlb'} \ar[r]^{{\rm FS}_{M}} &
       {\rm End}_{\ZM_{\ell}'M}(W) \ar[u]_{i_{P}}
     }$$
     where the left vertical map is
     % the action
     % of $\OC(Z^{1}(W_{F}^{0},\HG))$ on $i_{P}(W)$ through ${\rm FS}_{G}$ coincides with
     % the one obtained by inducing the action through ${\rm FS}_{M}$ and the natural map
     % $\OC(Z^{1}(W_{F}^{0},\HG))\To{}\OC(Z^{1}(W_{F}^{0},\HM))$ 
     given by pushforward of
     $1$-cocycles along $\HM\hookrightarrow\HG$.
   \item Compatibility with central characters : upon identifing $\widehat{Z_{M}}$ with
     $\HM_{\rm ab}$, the following diagram commutes :
     $$\xymatrix{ \Exc(W_{F},\HM)_{\Zlb'} \ar[r]^-{{\rm FS}_{M}} & \ZZz_{\Zlb'}(M) \\
       \Exc(W_{F},\HM_{\rm ab})_{\Zlb'} \ar[r]^-{{\rm FS}_{Z_{M}}} \ar[u] & \ZZz_{\Zlb'}(Z_{M}) \ar[u]}.$$
   \item Isomorphism for tori : the map ${\rm FS}_{Z_{M}}:\,\Exc(W_{F},\HM_{\rm ab})_{\Zlb'} \To{}
     \ZZz_{\Zlb'}(Z_{M})$ is an isomorphism.
   \item Continuity : for any finitely generated $\Zlb'M$-module $W$, the map $\Exc(W_{F},\HM)_{\Zlb'}\To{}{\rm
      End}_{\ZM_{\ell}'M}(W)$ factors over $\Exc(W_{F}^{0}/P_{F}^{e},\HM)_{\Zlb'}$ for some $e$.

   \end{enumerate}

Let us return to our representation $V=i_{P}(W)$ with $W$ a finitely generated cuspidal
$\Zlb' M$-module.
% Put $V':=V\otimes_{\Zlb}\Zlb'$ and $W':=V\otimes_{\Zlb}\Zlb'$. We still
% have $V'=i_{P}(W')$
% Of course, it is enough to prove that $V':=V\otimes_{\Zlb}\Zlb'=i_{P}(W')$ is 
% $\ZZz_{\ZM_{\ell}}(Z_{M})$
Since $W$ is $\ZZz_{\ZM_{\ell}'}(Z_{M})$-admissible, the last three properties
above show that $W$ is admissible over $\Exc(W_{F}^{0}/P_{F}^{e},\HM)_{\Zlb'}$ for some
$e\in\NM$. Since parabolic induction preserves admissibility, property (1) above
shows that, in order to prove that $V$ is $\ZZz$-finite (and thus conclude the proof of
Theorem \ref{thm_Z-admissible}), it suffices to prove that the map
$\Exc(W_{F}^{0}/P_{F}^{e},\HG)\To{}\Exc(W_{F}^{0}/P_{F}^{e},\HM)$ is finite.

\medskip

This leads us to our results on the dual side. Instead of excursion algebras, we will
first focus on the moduli space of Langlands parameters $Z^{1}(W_{F}^{0}/P_{F}^{e},\HG)$
over $\ZM[\frac 1p]$, as defined in \cite{DHKM} and its GIT quotient.
Let us pick a lift of Frobenius $\Fr$ in $W_{F}^{0}$ and let $\bar\Fr$ denote its image in $\rm{Aut}(\hat{G})$.
Our main result in this context is :
\begin{theorem}\label{thm_parameters_main}
 % Pick $\gamma\in W_{F}\setminus I_{F}$ and denote by $\bar\gamma$ its image in ${\rm Aut}(G)$.
  For any $e\geq 0$, the map $Z^{1}(W_{F}^{0}/P_{F}^{e},\HG)\To{}\HG$, $\varphi\mapsto
  \varphi(\Fr)$ induces a finite morphism $Z^{1}(W_{F}^{0}/P_{F}^{e},\HG)\sslash
  \HG\To{}\HG\rtimes\bar\Fr\sslash\HG$.
\end{theorem}

From this result, we will easily deduce the following corollary.

\begin{corollary}\label{cor_parameters_finiteness}
  Suppose $\HH\subset\HG$ is a closed reductive subgroup scheme of $\HG$ stable under
  $W_{F}$. Then
  \begin{enumerate}
  \item the natural morphism
    $Z^{1}(W_{F}^{0}/P_{F}^{e},\HH)\sslash\HH\To{}Z^{1}(W_{F}^{0}/P_{F}^{e},\HG)\sslash\HG$
    is finite.
  \item the natural morphism of algebras
  $\Exc(W_{F}^{0}/P_{F}^{e},\HG)_{\rm red} \To{} \Exc(W_{F}^{0}/P_{F}^{e},\HH)_{\rm red}$ is finite.
  \end{enumerate}
\end{corollary}

In (2) above, the subscript red denotes the reduced quotient. Since
$\ZZz_{\ZM_{\ell}'}(G)$ is a subring of $\ZZz_{\Ql}(G)$, it 
is reduced, by Bernstein's description. Hence the Fargues-Scholze morphism ${\rm FS}_{G}$ factors through the reduced
excursion algebra anyway, and this version of finiteness is sufficient to finish the proof
of Theorems \ref{thm_Z-admissible} and \ref{thm_finiteness}.

We also mention the following corollary of Theorem
\ref{thm_parameters_main}, which may have independent interest.

\begin{corollary}
  Let $F'$ be a finite extension of $F$ and put $P_{F'}^{e}:=P_{F}^{e}\cap W_{F'}$. Then
  the restriction morphism
  $Z^{1}(W_{F}^{0}/P_{F}^{e},\HG)\sslash\HG
  \To{}Z^{1}(W_{F'}^{0}/P_{F'}^{e},\HG)\sslash\HG$ is finite.
\end{corollary}

These results are proved in Section \ref{sec:finit-param-side}.  The proofs of the
finiteness results on the group side are given in Section \ref{sec:finit-group-side}.

\section{Finiteness on the parameters side}
\label{sec:finit-param-side}

In this section, $R$ denotes a ring of the form $\bar\ZZ[\frac 1N]$, where $N$ is any positive integer. Recall that any reductive
group scheme over $R$ is split (Prop A.13 of \cite{DHKM}). For a diagonalizable group scheme $D$
over $R$, we denote by $D^{0}$ its maximal subtorus (which may not be the maximal
connected subgroup scheme).
The following technical lemma makes working over $R$ a bit more convenient than working
over finite extensions of $\ZM[\frac 1N]$.
\begin{lemma}\label{lemma_tech_R}
  Let $D$ be a diagonalizable group scheme over $R$ acting on an $R$-scheme $X$, and let
  $X\To{\varphi}Y$ be a relative $D$-torsor for the fppf topology. Then the map $X(R)\To{}Y(R)$ is surjective.
\end{lemma}
\begin{proof}
  The fiber over a point $y\in Y(R)$ is a $D$-torsor over $S:=\Spec(R)$. So we need to show
  that any $D$-torsor over $S$ is trivial. Writing the character group  $X^{*}(D)$ as
  a product of cyclic abelian groups, we may assume that $D$ is  $\GG_{m}$ or $\mu_{m}$
  for some $m\in\NM^{*}$. In the case $D=\GG_{m}$, we
  have $H^{1}_{fppf}(S,\GG_{m})={\rm Pic}(S)={\rm colim}_{K\subset\overline\QQ}\,
  {\rm Pic}(\OC_{K}[\frac 1N])=\{1\}$ since, for any number fields $K\subset K'$, the map
  ${\rm Pic}(\OC_{K}[\frac 1N])\To{}{\rm Pic}(\OC_{K'}[\frac 1N])$ is trivial whenever
  $K'$ contains the Hilbert class field of $K$. In the case $D=\mu_{m}$, the exact
  sequence
  $ R^{\times}\To{(-)^{m}} R^{\times}\To{} H^{1}_{fppf}(S,\mu_{m})\To{}{\rm
    Pic}(R)\To{(-)^{m}}{\rm Pic}(R)$ and the surjectivity of $R^{\times}\To{(-)^{m}} R^{\times}$
  show that $H^{1}_{fppf}(S,\mu_{m})=\{1\}$ as desired.
\end{proof}

\begin{lemma} \label{lemma_Ch_St}
  Let $\HG$ be a reductive group scheme over $R$ and $\theta$ an automorphism of $\HG$ with
  finite order. Let $\HH$ be a reductive subgroup scheme of $\HG$ over $R$, and suppose it is
  ${\rm Int}_{g}\circ \theta$-stable for some $g\in \HG(R)$. Then the canonical morphism
  $ \HH g\rtimes\theta\sslash \HH\To{} \HG\rtimes\theta\sslash \HG$
  is finite.
\end{lemma}
\begin{proof}
  After maybe multiplying $g$ by some element $h\in \HH(R)$ on the left, we may assume that ${\rm
    Int}_{g}\circ\theta$ stabilizes a pinning $(T_{\HH},B_{\HH},X_{\HH})$ in $\HH$ ; indeed, the set
  of such $h$ is the set of $R$-points of a $Z(\HH)$-torsor, hence it is not empty
  by the above lemma.  Then, the
  ``twisted version'' of the Chevalley-Steinberg theorem
  (see   e.g. Prop 6.6 in \cite{DHKM}) implies that the inclusion $T_{\HH}\subset \HH$ induces a finite and
  surjective morphism $(T_{\HH})_{g\theta}\To{} \HH g\rtimes\theta\sslash
  \HH$, where $(T_{\HH})_{g\theta}$ denotes the co-invariants of $T_{\HH}$ under  ${\rm
    Int}_{g}\circ\theta$.
    Denoting by
  $S:=(T_{\HH})^{g\theta,0}$ the maximal subtorus of $T_{\HH}$ that is fixed under ${\rm
    Int}_{g}\circ\theta$, it follows that the inclusion $S\subset \HH$ also
  induces a finite surjective morphism $Sg\rtimes \theta\To{} \HH g\rtimes\theta\sslash \HH$.
  In particular, the ring $\mathcal{O}(\HH g\rtimes\theta)^{\HH}$, which is reduced,  embeds in
  $\mathcal{O}(Sg\rtimes\theta)$.
  % Observe that $\HG$, $\theta$, $\HH$, its pinning, and $g$ are all defined over a finitely generated
  % subring $R'$ of $R$. Since $\mathcal{O}(G\rtimes\theta)^{G}$ is a finitely generated
  % algebra 

  Therefore, if we can prove that $\mathcal{O}(Sg\rtimes\theta)$ is finite
  over $\mathcal{O}(\HG\rtimes\theta)^{\HG}$, we get that  $\mathcal{O}(\HH g\rtimes\theta)^{\HH}$
  is integral over $\mathcal{O}(\HG\rtimes\theta)^{\HG}$, hence also finite since
  $\mathcal{O}(\HH g\rtimes\theta)^{\HH}$ is a finitely generated $R$-algebra.
  % since $\mathcal{O}(G\rtimes\theta)^{G}$ is a finitely generated $R$-algebra\footnote{Technically,
  %   $\mathcal{O}(G\rtimes\theta)^{G}$ is not noetherian since $R$ is not noetherian, but
  %   all groups and morphisms involved here are base-changed from a finitely generated subring
  %   $R'$, over which any finitely generated algebra is noetherian.}, it suffices
 So we are left with proving that the map
  $ S g\rtimes \theta\To{} \HG\rtimes \theta\sslash \HG$ is finite.
  Let $n$ be the order of $\theta$, and consider the diagram
$$\xymatrix{
     S g\rtimes \theta\ar[r] \ar[d]_{(-)^{n}} & \HG\rtimes \theta\sslash \HG  \ar[d]^{(-)^{n}}
     \\
     S g_{n} \ar[r] & \HG \sslash \HG
}$$
Here the horizontal maps are the natural ones, and the vertical maps are induced by
raising to the power $n$ and we have put $g_{n}:=g\theta(g)\cdots\theta^{n-1}(g)$.
Note that, since $S$ is centralized by ${\rm Int}_{g}\circ\theta$, the left vertical map
is given by $sg\rtimes\theta\mapsto s^{n}g_{n}$. In particular it is a finite morphism,
hence the finiteness of the top horizontal map will follow if we can prove finiteness of the
bottom horizontal map.

Denote by $M:=C_{\HG}(S)$ the centralizer of $S$ in $\HG$, which is a Levi subgroup. Then we
have $g_{n}\in M(R)$, whence a factorization $Sg_{n}\To{} M\sslash M
\To{} \HG\sslash \HG$. The  second map is finite by the
Chevalley-Steinberg theorem. In order to prove that the first map is
finite, consider the isogeny $Z(M)^{0}\times M'\To{} M$ where
$Z(M)^{0}$ denotes the maximal central torus in $M$ and $M'$ the derived subgroup.
Another application of the
Chevalley-Steinberg theorem shows that the morphism
$(Z(M)^{0}\times M')\sslash M= Z(M)^{0}\times M'\sslash M\To{} M\sslash M$ is finite.
Thanks to the above lemma, we may write $g_{n}=zm'$ with $z\in Z(M)^{0}(R)$ and $m'\in
M'(R)$. Then we get a factorization $$ Sg_{n}\To{} Z(M)^{0}\times\{m'\} \To{} Z(M)^{0}\times M'\sslash M\To{}  M\sslash M$$
where the first two maps are closed immersions and the last one is finite. This implies
that $Sg_{n}\To{}\HG\sslash \HG$ is finite and completes the proof.
\end{proof}

We now fix a prime $p$ and assume further that $p|N$ and that the reductive group scheme $\HG$ over
$R=\bar\ZZ[\frac 1N]$ is endowed with a finite action of the Weil group $W_{F}$ of a local
field $F$ of residue characteristic $p$.
As usual, we denote by $\Fr$ a lift of Frobenius, and by $\bar\Fr$ its image in ${\rm Aut}(\HG)$.

\begin{theorem} \label{thm_parameters_side}
 % Pick $\gamma\in W_{F}\setminus I_{F}$ and denote by $\bar\gamma$ its image in ${\rm Aut}(G)$.
  For any $e\geq 0$, the map $Z^{1}(W_{F}^{0}/P_{F}^{e},\HG)\To{}\HG$, $\varphi\mapsto
  \varphi(\Fr)$ induces a finite morphism $Z^{1}(W_{F}^{0}/P_{F}^{e},\HG)\sslash
  \HG\To{}\HG\rtimes\bar\Fr\sslash\HG$.
\end{theorem}

Before giving the proof, let us draw some consequences.

\begin{corollary} \label{cor_parameters_side}
  Suppose $\HH\subset\HG$ is a closed reductive subgroup scheme of $\HG$ stable under
  $W_{F}$. Then the natural morphism
  $Z^{1}(W_{F}^{0}/P_{F}^{e},\HH)\sslash\HH\To{}Z^{1}(W_{F}^{0}/P_{F}^{e},\HG)\sslash\HG$ is finite.
\end{corollary}
\begin{proof}
  Consider the commutative diagram
  $$\xymatrix{Z^{1}(W_{F}^{0}/P_{F}^{e},\HG)\sslash\HG \ar[r] &
    \HG\rtimes\bar\Fr\sslash\HG \\
    Z^{1}(W_{F}^{0}/P_{F}^{e},\HH)\sslash\HH  \ar[r] \ar[u] &
    \HH\rtimes\bar\Fr\sslash\HH \ar[u] 
  }.$$
  The theorem says that the bottom horizontal map is finite, and Lemma \ref{lemma_Ch_St} says that the
  right vertical map is finite. It follows that the left vertical map is finite.
\end{proof}

Similarly, a $W_{F}$-equivariant isogeny $\HG\To{}\HG'$ induces a finite
morphism $Z^{1}(W_{F}^{0}/P_{F}^{e},\HG)\sslash\HG\To{}Z^{1}(W_{F}^{0}/P_{F}^{e},\HG')\sslash\HG'$.
Now, let us change $F$ instead of $\HG$.

\begin{corollary}
  Let $F'$ be a finite extension of $F$ and put $P_{F'}^{e}:=P_{F}^{e}\cap W_{F'}$. Then
  the restriction morphism
  $Z^{1}(W_{F}^{0}/P_{F}^{e},\HG)\sslash\HG
  \To{}Z^{1}(W_{F'}^{0}/P_{F'}^{e},\HG)\sslash\HG$ is finite.
\end{corollary}
\begin{proof} Let us choose the Frobenius lifts $\Fr\in W_{F}$ and $\Fr'\in W_{F'}$ such that $\Fr'=\Fr^{f}$,
  and consider the commutative diagram
  $$\xymatrix{Z^{1}(W_{F'}^{0}/P_{F'}^{e},\HG)\sslash\HG \ar[r] &
    \HG\rtimes\bar\Fr'\sslash\HG \\
    Z^{1}(W_{F}^{0}/P_{F}^{e},\HG)\sslash\HG  \ar[r] \ar[u]^{\rm restr.} &
    \HG\rtimes\bar\Fr\sslash\HG \ar[u]_{(-)^{f}} 
  }.$$
  Again the desired finiteness of the left vertical map  follows from  the finiteness of
  the bottom horizontal, provided by the theorem, and the finiteness of the right vertical
  map, which follows from the twisted version of the Chevalley-Steinberg map, since
  raising to power $f$ on a torus is a finite isogeny.
\end{proof}

\def\OC{\mathcal{O}}

Finally, we can also prove a  statement similar to that of
Theorem~\ref{thm_parameters_side} for the excursion
algebra
 $$\Exc(W_{F}^{0}/P_{F^{e}},\HG) = {\rm colim}_{n,F_{n}\rightarrow W_{F}^{0}/P_{F}^{e}}
  \mathcal{O}(Z^{1}(F_{n},\HG))^{\hat G}$$
  from Definition
  VIII.3.4 of \cite{FarguesScholze}. Indeed, we have a commutative diagram
  $$\xymatrix{
     \Exc(W_{F}^{0}/P_{F^{e}},\HG) \ar[r] & \OC(Z^{1}(W_{F}^{0}/P_{F}^{e},\HG))^{\HG} \\
\Exc(\langle\Fr\rangle,\HG) \ar[r]^-{\simeq} \ar[u] &  \OC(Z^{1}(\langle\Fr\rangle,\HG))^{\HG} =\OC(\HG\rtimes\Fr)^{\HG} \ar[u] 
  }$$
where the horizontal maps follow from the construction of the excursion algebra, and the
vertical maps are induced by restriction to the free abelian subgroup of
$W_{F}^{0}/P_{F}^{e}$ generated by $\Fr$. Note that, as with any free group, the bottom
horizontal map is an isomorphism. Theorem \ref{thm_parameters_side} tells us that the right vertical map is
finite. On the other hand, we know that the top horizontal map induces bijections on
$L$-valued points for any algebraically closed field $R$-algebra. In particular, this map
induces an embedding of the  \emph{reduced} excursion algebra
$\Exc(W_{F}^{0}/P_{F^{e}},\HG)_{\rm red}$ in
$\OC(Z^{1}(W_{F}^{0}/P_{F}^{e},\HG))^{\HG}$ (which was proved to be reduced in \cite{DHKM}).
Observe also that the
whole commutative square above is base-changed from the same square over some
finitely generated subring $R'\subset R$. Over such a subring,
the finitely generated $R'$-algebra  $ \Exc(\langle\Fr\rangle,\HG)$ is noetherian, so
we infer that the map $\Exc(\langle\Fr\rangle,\HG) \To{}
\Exc(W_{F}^{0}/P_{F^{e}},\HG)_{\rm red}$ is finite.
%By base-change, the same holds over $R$.

The following interesting corollary is  then proved like Corollary \ref{cor_parameters_side}.
\begin{corollary}
  Suppose $\HH\subset\HG$ is a closed reductive subgroup scheme of $\HG$ stable under
  $W_{F}$. Then the natural morphism of algebras
  $\Exc(W_{F}^{0}/P_{F}^{e},\HG)_{\rm red} \To{} \Exc(W_{F}^{0}/P_{F}^{e},\HH)_{\rm red}$ is finite.  
\end{corollary}

  % By construction, there is a canonical morphism
  % $$ \Exc(W_{F}^{0}/P_{F^{e}},\HG) \To{} \OC(Z^{1}(W_{F}/P_{F}^{e},\HG)^{\HG}.$$
  % By restriction to the free abelian group generated by $\Fr$, we also have a morphism
  % $$\Exc(\langle\Fr\rangle,\HG) \To{} \Exc(W_{F}/P_{F^{e}},\HG) .$$

% %  Denote by $PC(W_{F}/P_{F^{e}},\HG)$ the spectrum of the
% %  excursion algebra (PC for ``pseudo-characters''). 

\begin{proof}[Proof of Theorem \ref{thm_parameters_side}]
  \emph{Step 1 : reduction to the tame case.} We use the notation of Proposition 1.2 of
  \cite{DHKM}. This proposition provides us with a decomposition
  $$ \coprod_{(\phi,\tilde\varphi)} Z^{1}_{\rm
    Ad_{\tilde\varphi}}(W_{F}^{0}/P_{F},C_{\HG}(\phi)^{\circ})\sslash
  C_{\HG}(\phi)_{\tilde\varphi} \To\sim  Z^{1}(W_{F}^{0}/P_{F}^{e},\HG)\sslash   \HG.$$
 For any pair  $(\phi,\tilde\varphi)$ as above,  consider the diagram
  $$\xymatrix{
    Z^{1}(W_{F}^{0}/P_{F}^{e},\HG)\sslash\HG\ar[r]^{\varphi\mapsto\varphi(\Fr)\rtimes\bar\Fr} 
    & \HG\rtimes\bar\Fr\sslash\HG \\
    Z^{1}_{\rm Ad_{\tilde\varphi}}(W_{F}^{0}/P_{F},C_{\HG}(\phi)^{\circ})\sslash C_{\HG}(\phi)^{\circ}
    \ar[r] \ar[u]_{\eta\mapsto\eta.\tilde\varphi} &
    C_{\HG}(\phi)^{\circ}\tilde\varphi(\Fr)\rtimes\bar\Fr\sslash C_{\HG}(\phi)^{\circ} \ar[u]
  }$$
in which the bottom horizontal map is given by $\eta\mapsto
\eta(\Fr)\tilde\varphi(\Fr)\rtimes\bar\Fr$.
By the above decomposition, the left vertical map is a finite covering of a summand of
$Z^{1}(W_{F}^{0}/P_{F}^{e},\HG)\sslash\HG$. By Lemma \ref{lemma_Ch_St}, the right vertical map is
finite. So, if we can prove finiteness of the horizontal bottom map, then we get
finiteness of the horizontal top map restricted to the summand associated to
$(\phi,\tilde\varphi)$. Varying $(\phi,\tilde\varphi)$ we then get finiteness of the top map.
Since the action of ${\rm Ad}_{\tilde\varphi}(W_{F})$ on $C_{\HG}(\phi)^{\circ}$ is
finite, tamely ramified, 
and stabilizes a Borel pair, we are reduced to the ``tame case'' below. 

\medskip

\emph{Step 2 : the tame case.}  Suppose here that the action of $W_{F}$ on $\HG$ is trivial on $P_{F}$
and stabilizes a Borel pair $(\HB,\HT)$ of $\HG$. We are going to prove that the map
$$ Z^{1}_{\rm tame}:=Z^{1}(W_{F}^{0}/P_{F},\HG)\sslash\HG \To{} \HG\rtimes\bar\Fr\sslash\HG,\,\,
\varphi\mapsto \varphi(\Fr)\rtimes\bar\Fr$$
is finite. %Let $\Fr$ be a lift of Frobenius and
Let $s$ be a generator of $I_{F}^{0}/P_{F}$ (as a $\ZZ[\frac 1q]$-module).
We have
$$ Z^{1}(W_{F}^{0}/P_{F},\HG)=\left\{(\sigma,F)\in (\HG\rtimes\bar s)\times
(\HG\rtimes\bar\Fr),\, F\sigma F^{-1}=\sigma^{q}\right\}.$$ 

Let $\Omega:=N_{\HG}(\HT)/\HT$ denote the Weyl group of $\HG$ and $\pi
: N_{\HG}(\HT)\To{}\Omega$ the projection. 
Denote by $\Omega^{s}=\Omega^{s^{q}}$ the subgroup of %the Weyl group $\Omega=N_{\HG}(\HT)/\HT$
$\Omega$ fixed by $s$ (equivalently, by $s^{q}$), and by $N_{\HG}(\HT)_{s}$ its inverse image in
$N_{\HG}(\HT)$.
We further put
$$ N_{s}:=\{n\in N_{\HG}(\HT)_{s}, \, ns^{q}(n)^{-1}\in \HT^{s,0}\}.$$
% Then the map
% $$\beta : N_{\HG}(\HT)_{s} \To{} \HT\rtimes \Omega^{s}, \, n\mapsto \left(n s^{q}(n)^{-1},\pi(n)\right)$$
% defines a morphism of group schemes. We put
% $$N_{s}:=N_{\HG}(\HT)_{s}\times_{\HT\rtimes\Omega_{s}}(\HT^{s,0}\rtimes\Omega^{s}) =\beta^{-1}(\HT^{s,0}\rtimes\Omega^{s}).$$ 
This is a closed subgroup scheme of $N_{\HG}(\HT)_{s}$ whose intersection with $\HT$ is
the diagonalizable group scheme
$$ \HT\cap N_{s} = \{t\in \HT, \, ts^{q}(t)^{-1}\in \HT^{s,0}\}.$$

\def\dw{{\dot w}}

\begin{lemma}\label{lemma_ex_seq} With the foregoing notation,
  \begin{enumerate}
  \item the sequence
    $(\HT\cap N_{s})(R) \hookrightarrow N_{s}(R) \twoheadrightarrow \Omega_{s}$ is exact.
  \item $(\HT\cap N_{s})^{0}=\HT^{s,0}$ (recall that $^{0}$ denotes a maximal subtorus here).
  \end{enumerate}
\end{lemma}
\begin{proof}
 (1) Let us first prove that the morphism $N_{s}(R)\To{} \Omega_{s}$ is surjective.
Let $w\in\Omega_{s}$ and start with any lift
  $n$ of $w$ in $N_{\HG}(\HT)(R)$. Since $\HT^{s^{q},0}=\HT^{s,0}\To{}\HT_{s^{q}}$ is a finite covering, there exists
  $t_{s}\in \HT^{s,0}(R)$ whose image in $\HT_{s^{q}}(R)$ coincides with that of
  $ns^{q}(n)^{-1}$.
Applying Lemma \ref{lemma_tech_R} to the morphism $\HT\To{}\HT_{s^{q}}$,
we thus can find  $t\in \HT(R)$ such that $t_{s}=ts^{q}(t)^{-1}ns^{q}(n)^{-1}$. Then the
  element  $\dw:=tn$  satisfies $\dw s^{q}(\dw)^{-1}=t_{s}\in \HT^{s,0}(R)$, hence $\dw\in
  N_{s}(R)$ and maps to $w$.
  On the other hand, we see from the definitions that the fiber of $N_{s}$  over $w\in \Omega_{s}$
  is $(\HT\cap N_{s})(R)\dw$.

  (2) Consider the sequence of morphisms :
  $$\xymatrix{
    \HT^{s}=\HT^{s^{q}} \ar@{^{(}->}[r] &
    \HT\cap N_{s} \ar[rr]^{t\mapsto ts^{q}(t)^{-1}} &&
    \HT^{s,0} \ar@{->>}[r] &
    \HT_{s}=\HT_{s^{q}}
    }$$
  % $$ \HT^{s}=\HT^{s^{q}}\hookrightarrow \HT\cap N_{s} \To{t\mapsto t^{-1}s^{q}(t)}
  % \HT^{s,\circ}\twoheadrightarrow \HT_{s}=\HT_{s^{q}}.$$
It may not be exact but it is at least a complex. Since the last map is finite, this
implies that $(\HT\cap N_s)/\HT^s$ is finite, hence $\HT^{s,0}= (\HT\cap  N_{s})^{0}$.
% Obviously, we have $\HT^{s}\subset \HT\cap N_{s}$, whence the inclusion of maximal
%   tori $\HT^{s,0}\subset (\HT\cap  N_{s})^{0}$. On the other hand, if $t\in (\HT\cap N_{s})^{0}$, we have
%   $ts^{q}(t)^{-1} = s^{q}(t)s^{2q}(t)^{-1}$, hence
%   $(ts^{q}(t)^{-1})^{n}=t^{n}s^{qn}(t)^{-1}$ for all $n$. In particular, if $n$ is the
%   order of $s$, we get that $t^{n}\in \HT^{s^{q}}=\HT^{s}$. This shows that the maximal
%   torus of $(\HT\cap N_{s})^{\circ}$ is contained in $\HT^{s,0}$. 
\end{proof}

We are now interested in the closed subscheme $A$ of $\HT^{s,0}\times N_{s}$ defined on 
$R$-algebras by
$$A(R'):=\left\{(t,n)\in \HT^{s,0}(R')\times N_{s}(R'),\, n\Fr(t)n^{-1}t^{-q}=s^{q}(n)n^{-1}\right\}.$$
 The reason is that this condition is equivalent to asking
that $(t\rtimes s,n\rtimes \Fr)$ be a point of $Z^{1}(W_{F}^{0}/P_{F},\HG)$, so that $A$
identifies to a closed subscheme of $Z^{1}(W_{F}^{0}/P_{F},\HG)$.

\begin{lemma}
 $A$ is finite over $N_{s}$.
\end{lemma}
\begin{proof}
The morphism
$ \alpha:\, \HT^{s,0}\times N_{s} \To{}\HT^{s,0}\times N_{s}, \,\, (t,n)\mapsto (n\Fr(t)n^{-1}t^{-q},n)$
 is an endo-isogeny of the relative torus $\HT^{s,0}\times N_{s}$ over $N_{s}$. In particular, the
kernel $\ker(\alpha)$ of this isogeny is a finite  group scheme over $N_{s}$.
Now observe that $A$ is the preimage under the map $\alpha$ of the section $N_{s}\To{} \hat{T}^{s,0}\times
N_{s}$ given by $n\mapsto (s^{q}(n)n^{-1},n)$, so it is a torsor under
$\ker(\alpha)$, hence it is finite over $N_{s}$.
\end{proof}

Let us identify $A$ to a closed subscheme of $Z^{1}(W_{F}^{0}/P_{F},\HG)$ through
$(t,n)\mapsto (t\rtimes s,n\rtimes \Fr)$. We note that $A$ is then stable under
conjugation by $\HT^{s,0}$. We have therefore a commutative diagram

$$\xymatrix{
  Z^{1}(W_{F}^{0}/P_{F},\HG)\sslash\HG \ar[r]^-{\varphi\mapsto \varphi(\Fr)} &
  \HG\rtimes\Fr\sslash\HG \\
  A\sslash \HT^{s,0} \ar[u] \ar[r] &
  N_{s}\rtimes\Fr\sslash \HT^{s,0} \ar[u]  
  }$$

The last lemma shows that the horizontal bottom map is finite. Up to writing $N_{s}(R)$ as a
union of finitely many left orbits under $\HT^{s,0}(R)$ (thanks to lemmas \ref{lemma_ex_seq}
and \ref{lemma_tech_R}), Lemma \ref{lemma_Ch_St}
implies that the right vertical map is finite. 
By the next lemma, the left vertical map is
dominant. Hence $\OC(Z^{1}(W_{F}^{0}/P_{F},\HG))^{\HG}$, being reduced, is a
subalgebra of $\OC(A)^{\HT^{s,0}}$, and thus it is integral over
$\OC(\HG\rtimes\Fr)^{\HG}$ by what we have proved above.
But since it is also of finite type as an $\OC(\HG\rtimes\Fr)^{\HG}$-algebra (even as a
$R$-algebra), it is finite.
%Therefore, in order to conclude finiteness of the top
%horizontal map, it suffices (using noetheriannity) to prove that the left vertical map is
%dominant.
Therefore, the next lemma finishes the proof of the theorem.

\begin{lemma}\label{lemmaAfinoverN}
  The morphism %\coprod_{w\in\Omega^{s}} A_{w}\times \HT_{w}
  $A\To{}Z^{1}(W_{F}^{0}/P_{F},\HG)\sslash\HG $ is surjective on $L$ points for any
  algebraically closed field $R$-algebra.
\end{lemma}
\begin{proof}
  Recall that $L$-points of $Z^{1}(W_{F}^{0}/P_{F},\HG)\sslash\HG$ correspond to closed
  $\HG$-orbits in  $Z^{1}(W_{F}^{0}/P_{F},\HG)(L)$.
  So let $\varphi=(\sigma,F)\in  Z^{1}(W_{F}^{0}/P_{F},\HG(L))$
  have closed orbit under $\HG$. We must prove it is
  $\HG(L)$-conjugate to a point in $A(L)$.
  By a theorem of Richardson (see Theorem 4.13 of \cite{DHKM} in this context),  the subgroup
  $\varphi(W_{F}^{0})=\langle\sigma,F\rangle$ of 
  $\LG(L)$ is completely reducible, and this implies that the normal subgroup
  $\langle\sigma\rangle$ is also completely reducible. In turn, this implies that $\sigma$
  stabilizes a Borel pair of $\HG$, that we may assume to be $(\HB,\HT)$ after conjugating
  $\varphi$ (compare
  the proof of Prop. 4.19 (2) of \cite{DHKM}). Since $s$ also stabilizes $(\HB,\HT)$, it follows
  that $\sigma\in \HT(L)\rtimes s$. Using surjectivity of $\HT^{s,0}\To{}\HT_{s}$ as above,
  we can further conjugate $\varphi$ by an element of $\HT(L)$ to achieve $\sigma\in
  \HT^{s,0}(L)\rtimes s$.  Now, the (reduced) centralizer $\HG_{\sigma}$ is a possibly
  non-connected reductive algebraic subgroup of $\HG_{L}$ with maximal torus $\HT^{s,0}$
  and Borel subgroup $\HB^{s}$, and it is stable under conjugation by $F$.
  Moreover we have an isomorphism from $\HG_{\sigma}F$ to the (reduced) closed subscheme
  $Z^{1}(W_{F}^{0}/P_{F},\HG)_{L,\sigma}$ of $Z^{1}(W_{F}^{0}/P_{F},\HG)_{L}$ defined by the
  condition $\varphi(s)=\sigma$. The $\HG_{\sigma}$-orbit of $F$ in $\HG_{\sigma}F$ then
  identifies to the intersection of $Z^{1}(W_{F}^{0}/P_{F},\HG)_{\sigma}$ with
  the $\HG$-orbit of $(\sigma,F)$ in  $Z^{1}(W_{F}^{0}/P_{F},\HG)$. So it is closed,
  hence $\langle F\rangle$ is a completely reducible subgroup of $\HG_{\sigma}.\langle
  F\rangle$ and, as above, $F$ normalizes a Borel pair of $\HG_{\sigma}^{\circ}$.
  After conjugation by an element of $\HG_{\sigma}^{\circ}$, we may assume that $F$
  normalizes the Borel pair $(\HB^{s},\HT^{s,0})$.
%Let us pick $g_{\sigma}\in (\HG_{\sigma})^{\circ}(L)$ such that 
 % $F':=g_{\sigma}F$  normalizes the Borel pair $(\HB^{s},\HT^{s,0})$. In particular, $F'$
 Then $F$ also normalizes the centralizer of $\HT^{s,0}$ in $\HG$, which is $\HT$. So $F$ is of the
  form $n\rtimes\Fr$ for some $n\in N_{\HG}(\HT)$. Writing $\sigma=t\rtimes s$, we have
  the equality $n\Fr(t)n^{-1}(ns^{q}(n)^{-1})=t^{q}$. Since $\Fr$ and $F$ both normalize
  $\HT^{s,0}$, so does $n$, and this equality thus implies that $n\in N_{s}(L)$, and finally
  that $(\sigma,F)\in A(L)$. %Now our original $F$ is in $(\HG_{\sigma})^{\circ}F$.  
\end{proof}
\end{proof}

We close this section with the following corollary of the above
arguments. This result is stated for future reference but is not used
in this paper.

\begin{corollary}
  Let $R$ be either of the form $\bar\ZM[\frac 1N]$ as above or of the form $\Zl$ for some
  prime $\ell\neq p$. Then the canonical morphism
  $Z^{1}(W_{F}^{0}/P_{F}^{e},\HG)\To{}Z^{1}(W_{F}^{0}/P_{F}^{e},\HG)\sslash\HG$ is
  surjective on $R$-points.
\end{corollary}
\begin{proof}
  We may reduce to the tame case as in Step 1 of the proof of Theorem
  \ref{thm_parameters_side}. So we focus on the setting of Step 2 of that proof and use
  the same notation, including the closed subscheme $A$ of $Z^{1}(W_{F}^{0}/P_{F},\HG)$,
  which is stable under the action of the torus $\HT^{s,0}$. We have just proved that 
  the morphism $A\sslash \HT^{s,0}\To{}Z^{1}(W_{F}^{0}/P_{F},\HG)\sslash\HG$ is finite and
  surjective. Since $R$ is integrally closed and ${\rm Frac}(R)$ is
  algebraically closed, this morphism is thus surjective on $R$-points, and we are left to prove
  that the morphism $A\To{} A\sslash \HT^{s,0}$ is surjective on $R$-points.  It is
  certainly surjective on ${\rm Frac}(R)$-points.  So let $\bar  a$ be an $R$-point of
  $A\sslash \HT^{s,0}$ and let $\tilde a=(\tilde t,\tilde n)\in A({\rm Frac}(R))$ be a ${\rm Frac}(R)$-point
  above $\bar a$.  By Lemma \ref{lemma_ex_seq} (1), the morphism
  $N_{s}\To{} (N_{s}\rtimes\Fr)\sslash \HT^{s,0}$ is a disjoint union, indexed by
  $\omega\in\Omega^{s}$, of morphisms of the form $(\HT\cap N_{s})\dot\omega\To{}
  ((\HT\cap N_{s})\dot\omega\rtimes\Fr)\sslash\HT^{s,0}$. But each such morphism is a
  torsor over some quotient of $\HT^{s,0}$, namely $\HT^{s,0}/(\HT^{s,0})^{\dot\omega\Fr}$
  % the diagonalizable subgroup $({\rm id}-\dot\omega\Fr)(\HT^{s,0})$ of $\HT^{s,0}$,
  respectively.  In particular, the morphism $N_{s}\To{}
  (N_{s}\rtimes\Fr)\sslash \HT^{s,0}$ is surjective on $R$-points,  and its geometric
  fibers are $\HT^{s,0}$-orbits.
  % all  $\HT^{s,0}({\rm Frac}(R))$-orbits in $N_{s}({\rm Frac}(R))\rtimes\Fr$ are closed.
  This
  implies that we can conjugate $(\tilde t,\tilde n)$ by an element of $\HT^{s,0}({\rm
    Frac}(R))$ so that $\tilde n$ becomes $R$-valued. But now,  Lemma \ref{lemmaAfinoverN} tells us that
  the morphism $A\To{} N_{s}$, $a=(t,n)\mapsto n$ is finite, so this implies that $\tilde a$ is $R$-valued.
\end{proof}

\section{Finiteness on the group side}
\label{sec:finit-group-side}

%\subsection{Proof of the equivalence between Theorems \ref{thm_finiteness} and
%  \ref{thm_Z-admissible}}

We take up the setting of the introduction.

\begin{lemma} \label{lemma_Z_finite}
  Let $R$ be any noetherian ring with $p\in R^{\times}$. The full subcategory of
  $\Rep_{R}(G)$ formed by $\ZZz$-finite $RG$-modules
  % form a Serre subcategory $\Rep_{R}^{\ZZz-\rm fin}(G)$ of $\Rep_{R}(G)$.
  is closed under taking subobjects and quotients.
\end{lemma}
\begin{proof}
  Let us denote by $\ZZz_{V}$ the image of the map $\ZZz_{R}(G)\To{}{\rm End}_{RG}(V)$.
  If $V'$ is either a subobject or a quotient object of $V$,
  the kernel of the map $\ZZz_{R}(G)\To{}{\rm End}_{RG}(V')$
  contains the kernel of  $\ZZz_{R}(G)\To{}{\rm End}_{RG}(V)$, and therefore $\ZZz_{V'}$
  is a quotient of $\ZZz_{V}$.
  Suppose that $V$ is $\ZZz$-finite. In particular $\ZZz_{V}$ is a finitely
  generated $R$-algebra, so $\ZZz_{V'}$ has the same property. Moreover, $\ZZz_{V}$ is 
  noetherian, hence $V'$ is admissible over $\ZZz_{V}$ since $V$ is. Therefore $V'$ is $\ZZz$-finite.
\end{proof}

The next lemma justifies our claim in the introduction that Theorems \ref{thm_finiteness} and
  \ref{thm_Z-admissible} are equivalent.

\begin{lemma} \label{lemma_equivalence_main_thms}
  Let $R$ be any noetherian ring with $p\in R^{\times}$. The following are equivalent :
  \begin{enumerate}
  \item For all open compact subgroups of $G$, the Hecke algebra $R[H\backslash G/H]$ is finitely
    generated over its center, and its center is a finitely generated $R$-algebra.
  \item Any finitely generated $RG$-module is $\ZZz$-finite.
  \end{enumerate}
\end{lemma}
\begin{proof} Assume (2), let $H$ be an open compact subgroup and put $V:=R[G/H]$.
Let us identify ${\rm End}_{RG}(V)=R[H\backslash G/H]$, so that the action of  $\ZG_{R}(G)$ on $V$ is given
  by a morphism $\ZG_{R}(G)\To{} Z(R[H\backslash G/H])$, the image of which we denote by
  $\ZZz_{V}$, as above. Since $V$  is  finitely generated, it is $\ZG$-finite. In particular,
  $V^{H}=R[H\backslash G/H]$ is finitely  generated as a $\ZZz_{V}$-module, hence also as a
  $Z(R[H\backslash G/H])$-module. Moreover, $\ZZz_{V}$ being noetherian, $Z(R[H\backslash
  G/H])$ is also a finitely generated $\ZZz_{V}$-module, hence it is a finitely generated $R$-algebra.

  Now assume (1) and let us prove (2). Recall from  \cite[Appendix]{datfinitude} the
  decomposition  $\Rep_{R}(G)=\prod_{r}\Rep_{R}(G)_{r}$ according to depth, in the sense of Moy and
  Prasad.
  If $H$ is any open pro-$p$-subgroup of $G$, we thus have a canonical decomposition
  of $R[G/H]$ as a sum of mutually orthogonal subrepresentations
  $R[G/H]=\bigoplus_{r} R[G/H]_{r}$. Then the ring $R[H\backslash G/H]= {\rm
    End}_{RG}(R[G/H])$ and its center $Z(R[H\backslash G/H])$ decompose accordingly, and
  we have
  $$Z(R[H\backslash G/H])_{r}=Z({\rm End}_{RG}(R[G/H]_{r})).$$ 
  Let us fix a depth $r$. Since the factor subcategory $\Rep_{R}(G)_{r}$
  is generated by a finitely generated projective generator, we can find $H$
  such that $R[G/H]_{r}$ is such a (finitely generated) projective generator of
  $\Rep_{R}(G)_{r}$.  It follows that 
  the map $\ZG_{R}(G)\To{}{\rm End}_{RG}(R[G/H]_{r})$ induces an isomorphism
  $$\ZG_{R}(G)_{r}\To{\sim}Z({\rm End}_{RG}(R[G/H]_{r})).$$
  The two displayed isomorphisms above and our assumption (1) imply that $R[G/H]_{r}$ is
  $\ZG$-finite.
  Since it is true for all $r$, the last lemma implies that any finitely generated $V$ is $\ZG$-finite.  
\end{proof}

% \begin{lemma}
%   \label{lemma_reduc_Zl}
%   If Theorem \ref{thm_finiteness} holds for $R=\Zlb'$, then it holds for any noetherian
%   $\Zlb$-algebra $R$.
% \end{lemma}
\begin{lemma}
  \label{lemma_reduc_Zl}
Let $R$ be a noetherian ring with $p\in R^{\times}$ and let $R'$ be a noetherian $R$-algebra.
\begin{enumerate}
\item If Theorem \ref{thm_finiteness} holds for $R$, then it holds for $R'$.
\item If Theorem \ref{thm_finiteness} holds for $R'$ and $R'$ is faithfully flat over $R$, then it holds for $R$.
\end{enumerate}

\end{lemma}\begin{proof}
 (1)  Suppose   Theorem \ref{thm_finiteness} holds over $R$.  The isomorphism
$R'\otimes_{R}R[H\backslash G/H]\To\sim R'[H\backslash G/H]$
shows that $R'[H\backslash G/H]$ is a finitely generated module over
$R'\otimes_{R}Z(R[H\backslash G/H])$. It also induces a 
morphism
$R'\otimes_{R} Z(R[H\backslash G/H])\To{} Z(R'[H\backslash G/H])$
through which  the action of
$R'\otimes_{R}Z(R[H\backslash G/H])$ on $R'[H\backslash G/H]$ factors. Hence
$R'[H\backslash G/H]$ is a fortiori a finitely generated module over $Z(R'[H\backslash G/H])$.
Moreover, since $R'\otimes_{R}Z(R[H\backslash G/H])$ is noetherian, $Z(R'[H\backslash
G/H])$ is also a finitely generated $R'\otimes_{R}Z(R[H\backslash G/H])$-module, so
it is a finitely generated $R'$-algebra because $Z(R[H\backslash G/H])$ is a finitely
generated $R$-algebra by assumption.

(2) Since $R'$ is flat over $R$, the map $R'\otimes_{R} Z(R[H\backslash G/H])\To{}
Z(R'[H\backslash G/H])$ is an isomorphism. Since it is even faithfully flat, the finite
type property of $Z(R'[H\backslash G/H])$ as a $R'$-algebra implies that of $Z(R[H\backslash G/H])$
as an $R$-algebra, see \cite[\href{https://stacks.math.columbia.edu/tag/00QP}{Lemma 00QP}]{stacks-project}. For the same reason, the finite type property of $R'[H\backslash
G/H]$  
as a $Z(R'[H\backslash G/H])$-module implies that of $R[H\backslash G/H]$ as a
$Z(R[H\backslash G/H])$-module, see \cite[\href{https://stacks.math.columbia.edu/tag/03C4}{Lemma 03C4}]{stacks-project}.
\end{proof}

In the next lemma, we use our notation $\Zlb':=\Zlb[\sqrt q]$ from the
introduction. Recall that $i_{P}$ denotes normalized parabolic induction. Although at this
point normalization is not important, it will be more convenient in
the main argument of Section~\ref{subsec_proof_main}. 

\begin{lemma} \label{lemma_embedding_cuspidal}
  Let $Q$ be a finitely generated projective $\Zlb' G$-module. Then there is an embedding
  $Q\hookrightarrow V$ with $V$ of the form $V=\bigoplus_{P} i_{P}(W_{P})$ where $P$ runs
  among a finite set of parabolic subgroups of $G$, and each $W_{P}$ is a \emph{cuspidal and
    finitely generated} $\ell$-torsion free $\Zlb' M_{P}$-module.
\end{lemma}
\begin{proof}
  The $\Ql G$-module $Q\otimes\Ql$ is projective and finitely generated, so it is a direct
  factor of a finite direct sum of projective generators of Bernstein blocks of
  $\Rep_{\Ql}(G)$. By Bernstein's theory, such projective generators can be taken of the
  form $i_{P}(\pi\otimes\Ql[M/M^{0}])$, where
  \begin{itemize}
  \item $P$ is a parabolic subgroup of $G$ with Levi component $M$, 
  \item $\pi$ is a supercuspidal irreducible $\Ql M$-module, and
  \item $M^{0}$ is the subgroup of $M$ generated by compact elements.
  \end{itemize}
  Note that $\pi$ can be defined over a finite extension $E_{\pi}$ of $\Qlb$.  Moreover, we may and will
  choose $\pi$ so that the central character of $\pi$ takes values in $\Zl$, in which case, $\pi$
  admits an $M$-invariant lattice which can be defined over $\OC_{E_{\pi}}$. More
  precisely, there is an admissible and cyclic $\OC_{E_{\pi}}M$-module $L_{\pi}$ such that
  $\pi\simeq L_{\pi}\otimes_{\OC_{E_{\pi}}} \Ql$.

  So $Q\otimes\Ql$ is contained in a direct sum of representations of the form
  $i_{P}(L_{\pi}\otimes_{\Zlb}\Ql[M/M^{0}])$. Since $Q$ is finitely generated, it follows that there is a finite extension
  $E$ of $\Qlb$, containing $\sqrt q$ and such that $Q\otimes\Qlb$ is contained in a direct sum of representations of the form
  $i_{P}(L_{\pi}\otimes_{\Zlb} E[M/M^{0}])$. For the same reason, after maybe scaling the
  embedding by a power of $\ell$, we see that $Q$ is contained in a direct sum of representations of the form
  $i_{P}(L_{\pi}\otimes_{\Zlb} \OC_{E}[M/M^{0}])$. But  $L_{\pi}\otimes_{\Zlb}
  \OC_{E}[M/M^{0}]=\ind_{M^{0}}^{M}((L_{\pi})_{|M^{0}})$ is a finitely generated cuspidal
  $\Zlb M$-module, as desired.
\end{proof}

\subsection{Proof of Theorem \ref{thm_Z-admissible}} \label{subsec_proof_main}
We now start over the proof of Theorem \ref{thm_Z-admissible} as outlined in the
introduction. Thanks to Lemmas \ref{lemma_equivalence_main_thms} and \ref{lemma_reduc_Zl},
we may assume $R=\Zlb'$, and thanks to  Lemmas
 \ref{lemma_Z_finite} and 
\ref{lemma_embedding_cuspidal}, we are reduced to proving that if  $V$ is a $\Zlb'
G$-module of the
form $V=i_{P}(W)$ where $P=MU$ is a parabolic subgroup and $W$ is a finitely generated
cuspidal $\ell$-torsion free $\Zlb' M$-module, then $V$ is $\ZG$-finite. As in the introduction, denote by
$Z_{M}$ the maximal central torus of $M$. We know by \cite[Lemme 4.2]{datfinitude}
that $W$ is admissible as a $\ZG_{\Zlb'}(Z_{M}) M$-module.

Now we use Fargues and Scholze excursion theorem as explained in the introduction. We
refer to Definition  VIII.3.4 of \cite{FarguesScholze} for the excursion algebra and to 
\cite[Thm VIII.4.1]{FarguesScholze} and \cite[Thm IX.0.1]{FarguesScholze} for the
construction of the map ${\rm FS}_{G}$.  
The compatibility with parabolic induction (property (1) in the introduction) is proved in
\cite[Cor. IX.7.3]{FarguesScholze}. However, there, ordinary parabolic induction is used
while the map on the excursion side is twisted by a cyclotomic central cocycle. Using
normalized parabolic induction cancels out this cyclotomic twist. 
The compatibility ``with central characters''
(property (2)) follows
from  \cite[Thm IX.6.2]{FarguesScholze} and
  \cite[Thm IX.6.1]{FarguesScholze} applied to the isogeny $M_{\rm der}\times
  Z_{M}\longrightarrow M$. The fact that ${\rm FS}_{Z_{M}}$ is an  isomorphism (property
  (3)) follows from  \cite[Prop IX.6.5]{FarguesScholze}. Finally, the continuity property
  (property (4)) follows from \cite[Prop IX.5.1]{FarguesScholze} (and in fact is an
  important ingredient of the construction of the map ${\rm FS}_{G}$).

  Now, we return to our $V=i_{P}(W)$. The smooth $\Zlb' M$-module $W$ is
  admissible over $\ZG_{\Zlb'}(Z_{M})$, hence, by property (3) and property (2), it is admissible over
  $\Exc(W_{F},\HM)_{\Zlb'}$. Since induction preserves admissibility, it follows that $V$ is
  admissible as a $\Exc(W_{F},\HM)_{\Zlb'}G$-module. Moreover, property (4) tells us that the
  action of  $\Exc(W_{F},\HM)_{\Zlb'}$ on $W$ factors through some
  $\Exc(W_{F}^{0}/P_{F}^{e},\HM)_{\Zlb'}$. Since $W$ is $\ell$-torsion free, and since the
  nilradical of $\Exc(W_{F}^{0}/P_{F}^{e},\HM)_{\Zlb'}$ is $\ell$-torsion
  (because $\Exc(W_{F}^{0}/P_{F}^{e},\HM)_{\Zlb}[\frac 1\ell]\simeq
  \OC(Z^{1}(W_{F}^{0}/P_{F}^{e},\HM)_{\Zlb})^{\HM}[\frac 1\ell]$ which is reduced), this action
  actually factors through $\Exc(W_{F}^{0}/P_{F}^{e},\HM)_{\Zlb',\rm red}$.
  It then follows from property (1) that there is $e\in\NM$ such that the action of $\Exc(W_{F},\HG)_{\Zlb'}$ on $V$
  factors through $\Exc(W_{F}^{0}/P_{F}^{e},\HG)_{\Zlb', \rm red}$. We can now apply Corollary
  \ref{cor_parameters_finiteness}, which says that  $\Exc(W_{F}^{0}/P_{F}^{e},\HM)_{\Zlb',\rm red}$ is finite over
  $\Exc(W_{F}^{0}/P_{F}^{e},\HG)_{\Zlb', \rm red}$.  This implies that $V$ is admissible over
  $\Exc(W_{F}^{0}/P_{F}^{e},\HG)_{\Zlb'}$ and, a fortiori, that it is admissible over
  $\ZG_{\Zlb'}(G)$.

  It remains to prove that the image $\ZG_{V}$ of the map $\ZG_{\Zlb'}(G)\To{}\End_{\Zlb' G}(V)$
  is a finitely generated $\Zlb'$-algebra. Denote by $\EC_{V}\subset \ZG_{V}$ the image of
  $\Exc(W_{F},\HG)_{\Zlb'}$ in $\End_{\Zlb' G}(V)$. We know that $\EC_{V}$ is a finitely generated
  $\Zlb'$-algebra, hence it is noetherian and it suffices to prove that  $\End_{\Zlb' G}(V)$ is
  a finitely generated module over $\EC_{V}$. Note that  $\End_{\Zlb' G}(V)
  = \End_{\EC_{V}G}(V)$.
  Suppose we can find an open compact subgroup $H$ such that the restriction map
  $\End_{\EC_{V} G}(V)\To{}\End_{\EC_{V}}(V^{H})$ is injective, then we are done since  
  $V^{H}$ is a finitely generated $\EC_{V}$-module. In order to find such an $H$, observe
    that $V$ belongs to some bounded depth category
  $\Rep_{\Zlb'}(G)_{\leq r}$, since $V\otimes\Qlb$ is finitely generated and contains $V$, which is
  $\ell$-torsion free. So it suffices to pick $H$ such that some finitely generated
  projective generator of $\Rep_{\Zlb'}(G)_{\leq r}$ is generated by its
  $H$-invariants. Indeed, the $\Zlb' G$-module $V$ is then generated by $V^{H}$, so the
  restriction map $\End_{\EC_{V} G}(V)\To{}\End_{\EC_{V}}(V^{H})$ is injective.

  \medskip

  We end this subsection by noting that the above proof actually shows the following :

  \begin{corollary}  \emph{(of the proof)} \label{cor_Z_finite_over_E}
    \begin{enumerate}
    \item Any finitely generated smooth $\Zlb' G$-module is admissible
      over   $\Exc(W_{F},\HG)_{\Zlb'}$ (through ${\rm
        FS}_{G}$).
    \item For any $r>0$, there is $e\in\NM$ such that the composition  $\Exc(W_{F},\HG)
      \To{} \ZG_{\Zlb'}(G)_{\leq r}$
      % of ${\rm FS}_{G}$ with the projection on  the depth $\leq r$ summand
      factors  over $\Exc(W_{F}^{0}/P_{F}^{e},\HG)_{\rm red}$
      and makes $\ZG_{\Zlb'}(G)_{\leq r}$ a finite module over
      $\Exc(W_{F}^{0}/P_{F}^{e},\HG)_{\Zlb',\rm red}$.
    \end{enumerate}
    
  \end{corollary}

 We also have the following easy converse to Theorem
 \ref{thm_Z-admissible}.
 \begin{remark}
   Let $V$ be a $\ZG$-finite  $RG$-module of bounded depth. Then $V$
   is finitely generated.
 \end{remark}
 \begin{proof}
   It suffices to prove this when $V\in\Rep_{R}(G)_{r}$ for some
   depth $r$. Let $P_{r}$ be a finitely generated projective generator of
   $\Rep_{R}(G)_{r}$. It is generated by its $H$-invariants for a sufficiently small open
   pro-$p$-subgroup. Therefore, any object in $\Rep_{R}(G)_{r}$ is also generated by its
   $H$-invariants. In particular $V$ is generated by $V^{H}$, and it follows that any
   generating set of $V^{H}$ as a  $R[H\backslash G/H]$-module is a generating set as an
   $RG$-module.
   But our assumption says that $V^{H}$ is a  finite $\ZG_{R}(G)$-module, hence it is a fortiori
   a finite $R[H\backslash G/H]$-module. 
 \end{proof}

% Finally, before turning to second adjointness, we mention the following consequence of the
% finiteness properties above, that allows us to  relax some hypothesis in the literature.

% \begin{corollary}
%   Let $R$ be a noetherian $\ZM_{\ell}$-algebra and $H$ an open compact subgroup of
%   $G$. There is an open compact-mod-center subset $\Omega_{H,R}\subset G$ such that any
%   cuspidal function in $R[H\backslash G /H]$ is supported in $\Omega_{H,R}$. 
% \end{corollary}
% \begin{proof}
%   We know that  $R[G/H]$ is $\ZG$-finite, hence any sub-$RG$-module of $R[G/H]$ is also
%   $\ZG$-finite. In particular, its maximal cuspidal submodule is $\ZG$-finite, hence it is
%   finitely generated by the previous remark. But then we know that it is $R
%   Z_{G}$-admissible by \cite[Lemma 4.2]{datfinitude}. In particular, the space of cuspidal functions in $R[H\backslash G
%   /H]$ is a finite $R Z_{G}$-module, which implies the claim. 
% \end{proof}

% \begin{corollary} Generic irreducibility holds for any parabolic subgroup of $G$.
%   % Let $k$ be an algebraically closed field over $\ZM[1/p]$, let $P$ be a parabolic
%   % subgroup of $G$, let $M$ be a Levi subgroup of $P$, and let $\pi$ be an irreducible
%   % $kM$-module. Then, for a dense subset of unra 
% \end{corollary}
% \begin{proof}
% We refer to Theorem 5.1 of \cite{Datnu} for the precise statement of generic irreducibility.
%  The proof given there assumes the existence of cocompact lattices in $G$. But the
%  subsequent remark there gives an alternative argument under an hypothesis (H) that we have just
%  proved above.
% \end{proof}

\section{Second Adjointness} \label{sec:consequences}

We now study consequences of the above results for parabolic induction and restriction.  Let $R$ be a noetherian $\ZZ[\frac{1}{p}]$-algebra, and let $P=MU$ be a parabolic subgroup of $G$.  It is an easy consequence of the Bernstein-Deligne description of $\ZG_{\overline{\QQ}}(G)$ in \cite{BD} that one has a unique map:
$$\ZG_{\QQ}(G) \rightarrow \ZG_{\QQ}(M)$$
such that for any smooth $\QQ M$-module $V$, one has a commutative diagram:
     $$\xymatrix{
       \ZG_{\QQ}(G) \ar[r] \ar[d]
       & {\rm End}_{\QQ G}(I_{P}V)
       \\
       \ZG_{\QQ}(M) \ar[r] &
       {\rm End}_{\QQ M}(V) \ar[u]_{I_{P}}.
     }$$

In fact, it is not hard to deduce a similar result over $\ZZ[\frac{1}{p}],$ or indeed over any flat $\ZZ[\frac{1}{p}]$-algebra $R$:

\begin{theorem} \label{thm_bernstein_induction}
Let $R$ be a noetherian flat $\ZZ[\frac{1}{p}]$-algebra.  Then there is a unique map $\ZG_R(G) \rightarrow \ZG_R(M)$ such that for any smooth $R M$-module $V$, one has a commutative diagram:
     $$\xymatrix{
       \ZG_R(G) \ar[r] \ar[d]
       & {\rm End}_{R G}(I_{P}V)
       \\
       \ZG_R(M) \ar[r] &
       {\rm End}_{R M}(V) \ar[u]_{I_{P}}.
     }$$
\end{theorem}
\begin{proof}
We first prove this for $R = \ZZ[\frac{1}{p}]$. Recall from \cite[Thm 5.2]{MoyPrasad} and \cite[5.12]{Vig96} that the parabolic induction and restriction functors preserve depth, so that for each depth $r$, the map $\ZG_{\QQ}(G) \rightarrow \ZG_{\QQ}(M)$ takes $\ZG_{\QQ}(G)_r$ into $\ZG_{\QQ}(M)_r$. We begin by showing that this map takes $\ZG_{\ZZ[\frac{1}{p}]}(G)_r$ to $\ZG_{\ZZ[\frac{1}{p}]}(M)_r$.  To do so we identify $\ZG_{\ZZ[\frac{1}{p}]}(M)_r$ with the center of the endomorphism ring of $\ZZ[\frac{1}{p}][M/H]_r$ for some sufficiently small subgroup $H$ of $M$.  Note that the action of $z \in \ZG_{\ZZ[\frac{1}{p}]}(G)_r$ on $I_P \QQ[M/H]_r$ is given by $I_P z_M$ for some $z_M \in \ZG_{\QQ}(M)_r$, and it suffices to show that $z_M$ lies in $\ZG_{\ZZ[\frac{1}{p}]}(M)_r$; that is, that it preserves the $\ZZ[\frac{1}{p}][M]$-submodule $\ZZ[\frac{1}{p}][M/H]_r$ of $\QQ[M/H]_r$.  Certainly for a sufficiently large integer $a$ not divisible by $p$, the element $a z_M$ of $\ZG_{\QQ}(M)_r$ lies in $\ZG_{\ZZ[\frac{1}{p}]}(M)_r$, and the action of $a z_G$ on $I_P \Zlb[M/H]_r$ is given by $I_P a z_M.$

On the other hand, for any prime $\ell \neq p$, the map 
$$I_P: \End_{\FF_{\ell} M}(\FF_{\ell}[M/H]_r) \rightarrow \End_{\FF_{\ell} G}(I_P \FF_{\ell}[M/H]_r)$$
is injective, as $I_P$ is faithful.  This implies that if $x$ is an endomorphism of $\ZZ[\frac{1}{p}][M/H]_r$ such that $I_P x$ is divisible by $\ell$ in the endomorphism ring of $I_P \ZZ[\frac{1}{p}][M/H]_r,$ then $x$ is divisible by $\ell$ in the endomorphism ring of $\ZZ[\frac{1}{p}][M/H]_r.$  In particular it follows that $z_M$ gives rise to a (necessarily central) endomorphism of $\ZZ[\frac{1}{p}][M/H]_r,$ and thus that $z_M$ lies in $\ZG_{\ZZ[\frac{1}{p}]}(M).$

Now let $V$ be a smooth $\ZZ[\frac{1}{p}] M$-module, and let $V_r$ denote its depth $r$ summand, for each $r$.  It suffices to verify the commutativity of the diagram for each $r$, and thus, since $V_r$ admits a projective resolution by direct sums of copies of $\ZZ[\frac{1}{p}][M/H]_r$, it suffices to check the commutativity of the diagram in the case $V = \ZZ[\frac{1}{p}][M/H]_r.$  In this case the commutativity of the diagram is a direct consequence of the construction in the previous paragraphs.

We finally turn to the case of $R$ a flat $\ZZ[\frac{1}{p}]$-algebra.  In such a case we have natural isomorphisms:
$$\ZG_R(M)_r \cong Z(\End_{R M}(R[M/H]_r)) \cong Z(\End_{\ZZ[\frac{1}{p}] M}(\ZZ[\frac{1}{p}][M/H]_r) \otimes_{\ZZ[\frac{1}{p}]} R)$$
and since $R$ is flat the right hand side is isomorphic to $\ZG_{\ZZ[\frac{1}{p}]}(M) \otimes_{\ZZ[\frac{1}{p}]} R$.  We have a similar morphism with $G$ in place of $M$, and we take the map $\ZG_R(G) \rightarrow \ZG_R(M)$ to be the one obtained, via these identifications, by base change of the map $\ZG_{\ZZ[\frac{1}{p}]}(G) \rightarrow \ZG_{\ZZ[\frac{1}{p}]}(M).$  The commutativity of the diagram can be verified as above, by reducing to the case of $V = R[M/H]_r,$ where it is clear.
\end{proof}

Note that if the ring $R$ contains a square root of $q$ (whose choice must be fixed), then one has an exactly analogous result with normalized parabolic induction in place of unnormalized.

If $R$ is not flat over $\ZZ[\frac{1}{p}]$, then it is not clear that the natural map $\ZG_{\ZZ[\frac{1}{p}]}(G) \otimes_{\ZZ[\frac{1}{p}]} R \rightarrow \ZG_R(G)$ is surjective (although we expect that it is).  For such $R$ one certainly has a map $\ZG_{\ZZ[\frac{1}{p}]}(G) \otimes_{\ZZ[\frac{1}{p}]} R \rightarrow \ZG_R(M)$ making the diagram of Theorem~\ref{thm_bernstein_induction} commute, but it is not clear that this map descends to a map from $\ZG_R(G)$ to $\ZG_R(M)$.

When $R$ is a $\Zlb$-algebra for some $\ell$, we can combine Theorem~\ref{thm_bernstein_induction} with the results of the previous section to prove:

\begin{lemma}
Let $R$ be a noetherian $\Zlb$-algebra.  For each $r$, the ring $\ZG_R(G)_r$ is finitely generated as a $\ZG_{\Zlb}(G)_r \otimes_{\Zlb} R$-module.
\end{lemma}
\begin{proof}
We identify $\ZG_R(G)_r$ with the center of the endomorphism ring of $R[G/H]_r$ for some sufficiently small subgroup $H$ of $G$.  Since $\Zlb[G/H]_r$ is admissible over $\ZG_{\Zlb}(G)_r$ by Theorem \ref{thm_Z-admissible}, $R[G/H]_r$ is admissible over $\ZG_{\Zlb}(G)_r \otimes_{\Zlb} R$, and the result follows.
\end{proof}

Combining this with Theorem~\ref{thm_bernstein_induction} we obtain:

\begin{theorem} \label{thm_bernstein_finiteness}
Let $R$ be a noetherian $\Zlb$-algebra.
The map $\ZG_{\Zlb}(G)_r \otimes_{\Zlb} R \rightarrow \ZG_R(M)_r$ makes $\ZG_R(M)_r$ into a finitely generated $\ZG_{\Zlb}(G)_r \otimes_{\Zlb} R$-module.
\end{theorem}
\begin{proof}
  The previous lemma allows us to reduce to the case $R = \Zlb.$ In this case, the claim
  follows from Corollary \ref{cor_Z_finite_over_E} and Corollary
  \ref{cor_parameters_finiteness} (2).

\end{proof}

With these results in hand we now turn to the question of establishing second adjointness.  The surprising thing is that in spite of relying on the finiteness theorems of the previous sections, which require one to work over $\Zlb$ for some $\ell$, our results suffice to establish second adjointness over $\ZZ[\frac{1}{p}]$.  Our basic approach is closely related to that used to establish second adjointness in section 11 of~\cite{HelmBC}. We begin by showing:

\begin{lemma} \label{lem_admissible_restriction}
Let $P = MU$ be a parabolic subgroup of $G$, and $H$ a compact open subgroup of $G$.  Then the parabolic restriction $R_P \Zlb[G/H]_r$ is admissible over $\ZG_{\Zlb}(G)$.
\end{lemma}
\begin{proof}
It is an easy consequence of the Iwasawa decomposition that the parabolic restriction of any finitely generated $\Zlb[G]$-module is finitely generated as a $\Zlb[M]$-module; in particular $R_P \Zlb[G/H]_r$ is finitely generated as a $\Zlb[M]$-module, and hence admissible over $\ZG_{\Zlb}(M)$.  Since $\ZG_{\Zlb}(M)$ is finitely generated as a module over $\ZG_{\Zlb}(G)$, the result follows.
\end{proof}

We will use this lemma to establish ``stability'' for the modules $\Zlb[G/H]_r,$ and later deduce stability results for an arbitrary object of $\Rep_{\ZZ[\frac{1}{p}]}(G)_r.$  Let us recall what this means.  Fix $P = MU$ and $\overline{P} = M\overline{U}$ parabolic subgroups of $G$, such that $P \cap \overline{P} = M$, let $\lambda$ be a totally positive central element of $M$, and let $K$ a compact open subgroup of $G$ that is decomposed with respect to $P,\overline{P}$; that is, such that $K = K^- K_M K^+$, where $K_M = K \cap M$, $K^+ = K \cap U$, and $K^- = K \cap \overline{U}$.   Recall that a $\Zlb G$-module $V$ is called $K,P$-stable if there exists an integer $c_{K,P,\lambda} \geq 1$, and a direct sum decomposition $V^K = V^K_0 \oplus V^K_{*}$ such that:
\begin{itemize}
    \item $T_{\lambda} = 1_{K \lambda K}$ acts invertibly on $V^K_{*}$, and
    \item $T_{\lambda}^{c_{K,P,\lambda}}$ annihilates $V^K_0$.
\end{itemize}
We will call such an integer $c_{K,P,\lambda}$ a {\em constant of $K,P,\lambda$-stablility} for $V$, or simply a {\em constant of stability} if we wish to supress the dependence on $K$, $P$ and $\lambda.$  Note that these constants depend on $\lambda$, although the notion of $K,P$-stablility is independent of the choice of $\lambda.$

The key point is the following, 

\begin{lemma} \label{lem_stablility} Let $R$ be a noetherian $\Zlb$-algebra, and let $V$ be an admissible $R G$-module such that $R_P V$ is an admissible $R M$-module.  Then $V$ is $K,P$-stable.
\end{lemma}
\begin{proof}
This is proven in the case $G = \GL_n(F)$ in~\cite{HelmBC}, Lemmas 11.12 and 11.13.  The proof carries over, with only minor changes such as the use of non-normalized parabolic restriction in place of normalized, to the current setting.
\end{proof}

In particular we immediately deduce from Lemma~\ref{lem_admissible_restriction} and the above that:

\begin{corollary} For all $H$ and $r$, $\Zlb[G/H]_r$ is $K,P$-stable for any pair $K,P$ such that $K$ is decomposed with respect to $P$.
\end{corollary}

We will need tools to deduce $K,P$-stablility of other $\ZZ[\frac{1}{p}] G$-modules from that of $\Zlb[G/H]_r$.  We first observe that if $V$ and $W$ are both $K,P$-stable, and $f: V \rightarrow W$ is any morphism, then $f$ maps $V_0$ into $W_0$ and $V_*$ into $W_*$.  In particular any endomorphism of a $K,P$-stable representation $V$ preserves the direct sum decomposition $V^K = V^K_0 \oplus V^K_*$.

From this one immediately deduces that if $f: V \rightarrow W$ is a map of $K,P$-stable modules, then the kernel and cokernel of $f$ are also $K,P$-stable, and the constants of $(K,P,\lambda)$-stability of the kernel and cokernel of $f$ are bounded by those of $V$ and $W$, respectively.  Moreover, if $\{V_i\}_{i \in I}$ is a collection of $\ZZ[\frac{1}{p}] G$-modules that are $K,P$-stable, and the constants of stability of the $V_i$ are uniformly bounded by a constant $c$, then the direct sum and product of the $V_i$ is also $K,P$-stable, with $c$ a constant of stability.

Since every object of $\Rep_{\Zlb}(G)_r$ has a projective resolution by direct sums of copies of $\Zlb[G/H]_r$, it follows that every object of $\Rep_{\Zlb}(G)_r$ is $K,P$-stable, and that if $c_{\ell}$ is a constant of $(K,P,\lambda)$-stability for $\Zlb[G/H]_r$, then it is also a constant of $(K,P,\lambda)$-stability for every object of $\Rep_{\Zlb}(G)_r$.

From this it is not hard to deduce:

\begin{lemma} \label{lem_uniform_constant}
Fix $K$, $P$, and $\lambda$ as above.  There exists an absolute constant $c$ such that for any prime $\ell \neq p$, and any object $V$ of $\Rep_{\Zlb}(G)_r$, the constant $c$ is a constant of $(K,P,\lambda)$-stability for $V$.
\end{lemma}
\begin{proof}
Let $c_{\ell}$ be the smallest positive integer that is a constant of $(K,P,\lambda)$-stability for $\Zlb[G/H]_r$; the discussion of the previous paragraph shows in particular that $c_{\ell}$ is then also a constant of stability for $\overline{\Qlb}[G/H]_r$.  Conversely, since $\Zlb[G/H]_r \subseteq \Qlb[G/H]_r,$ we have that $c_{\ell}$ is the smallest constant of stability for $\overline{\Qlb}[G/H]_r$.  On the other hand, fixing an isomorphism of $\overline{\Qlb}$ with ${\mathbb C}$, and noting that $K,P$-stability is a purely algebraic notion, we find that $c_{\ell}$ is the smallest constant of stability of ${\mathbb C}[G/H]_r$.  In particular $c_{\ell}$ does not depend on $\ell$, and the claim follows.
\end{proof}

We now introduce a duality on $\Rep_{\Zlb}(G).$  For $V$ a smooth $\Zlb G$-module, let $V^{\vee,\ell}$ denote the $G$-smooth elements of $\Hom_{\Zlb}(V, \Qlb/\Zlb).$  Note that we have natural isomorphisms:
$$\Hom_{\Zlb G}(V, W^{\vee,\ell}) \cong \Hom_{\Zlb}(V \otimes_{\Zlb G} W, \Qlb/\Zlb) \cong \Hom_{\Zlb G}(W, V^{\vee,\ell});$$
in particular the functor $V \mapsto V^{\vee,\ell}$ takes projectives to injectives.  Moreover, the ``double dual'' map $V \mapsto (V^{\vee,\ell})^{\vee,\ell}$ is an isomorphism for $V$ a simple $\Zlb G$-module.  From this one deduces that (for $H$ compact open and sufficiently small) $\Zlb[G/H]_r^{\vee,\ell}$ is an injective $\Zlb G$-module that admits every simple object of $\Rep_{\Zlb}(G)_r$ as a submodule.  Since the ``forgetful functor'' from $\Zlb G$-modules to $\ZZ[\frac{1}{p}]$-modules is right adjoint to the exact functor $V \mapsto V \otimes_{\ZZ[\frac{1}{p}]} \Zlb$, we find that $\Zlb[G/H]_r^{\vee,\ell}$ remains injective as a $\ZZ[\frac{1}{p}] G$-module.  We will denote this module by $I_{\ell}$.

\begin{lemma} \label{lem_resolution}
Let $V$ be a simple object of $\Rep_{\ZZ[\frac{1}{p}]}(G)_r$.  Then $V$ embeds in $I_{\ell}$ for some $\ell$.  In particular, every object of $\Rep_{\ZZ[\frac{1}{p}]}(G)_r$ admits an injective resolution by direct products of copies of $I_{\ell}$ for varying $\ell$.
\end{lemma}
\begin{proof}
If $V$ is killed by multiplication by $\ell$ for some $\ell$ then we can regard $V$ as a $\Zlb G$-module and the claim is clear.  Otherwise multiplication by $\ell$ is invertible on $V$ for all $\ell$ and we can regard $V$ as a $\QQ G$-module. Then $V$ embeds in $V \otimes_{\QQ} \Qlb$, and the latter embeds in a direct product of copies of $I_{\ell}.$  We thus obtain an embedding of $V$ in a product of copies of $I_{\ell}$; the projection of $V$ to at least one of these copies is nonzero and thus injective.

To prove the final claim it suffices to show that every object of $\Rep_{\ZZ[\frac{1}{p}]}(G)_r$ embeds in a direct product of copies of $I_{\ell}$ for varying $\ell$.  Fix an object $V$ of this category, and for each pair of subobjects $(W,W')$ of $V$ such that $W$ is contained in $W'$ and the quotient $W'/W$ is simple, fix an embedding $\iota_{W,W'}$ of $W'/W$ into an injective $I_{W,W'}$ that is isomorphic to $I_{\ell}$ for some $\ell \neq p.$  Then we may regard $\iota_{W,W'}$ as a map from $W'$ to $I_{W,W'}$, and injectivity of $I_{W,W'}$ allows us to extend this map to a map: ${\hat \iota}_{W,W'}: V \rightarrow I_{W,W'}$.  It suffices to show that the product of the maps ${\hat \iota}_{W,W'}$ is an embedding of $V$ in $\prod_{W,W'} I_{W,W'}$; suppose otherwise.  Then the kernel of this product contains a nonzero, finitely generated subobject $W'$; this subobject admits a simple quotient, so we may fix a further subobject $W$ of $W'$ with $W'/W$ simple.  But the map ${\hat \iota}_{W,W'}$ is nonzero on $W'$, contradicting the fact that $W'$ is contained in the kernel of the product map.
\end{proof}

\begin{corollary} \label{cor_stability}
Every object of $\Rep_{\ZZ[\frac{1}{p}]}(G)_r$ is $K,P$-stable.
\end{corollary}

For $V$ a smooth $\ZZ[\frac{1}{p}]G$-module,  define $V^{\vee}$ to be the set of $G$-smooth elements of $\Hom_{\ZZ[\frac{1}{p}]}(V, \QQ/\ZZ[\frac{1}{p}])$.  This gives a faithful duality functor on $\Rep_{\ZZ[\frac{1}{p}]}(G)_r$.  As with the latter functors $V \mapsto V^{\vee,\ell}$ introduced earlier, the functor $V \mapsto V^{\vee}$ takes projectives to injectives and the ``double dual'' map $V \mapsto (V^{\vee})^{\vee}$ is an isomorphism for $V$ a simple $\ZZ[\frac{1}{p}]$-module.  Using this duality and our stability results, we can prove the following version of Jacquet's lemma for arbitrary objects of $\Rep_{\ZZ[\frac{1}{p}]}(G)_r:$

\begin{lemma} \label{lem:jacquet}
Let $V$ be an object of $\Rep_{\ZZ[\frac{1}{p}]}(G)_r$, and $P = MU$, $\overline{P} = M\overline{U}$ parabolic subgroups of $G$ with $P \cap \overline{P} = M$.  There is a natural isomorphism:
$$R_P (V^{\vee}) \cong (R_{\overline{P}}V)^{\vee}$$
where $\overline{P}$ denotes the opposite parabolic to $P$.
\end{lemma}
\begin{proof}
The proof follows the basic strategy of Bernstein, using $K,P$-stability of $V$ for a cofinal family of compact open subgroups $K$ of $G$; for $\GL_n(F)$ the details can be found in the proof of Lemma 11.16 of~\cite{HelmBC}, which carries over to the current setting with only minor changes, such as replacing the duality over $\Zlb$ with that over $\ZZ[\frac{1}{p}]$, and normalized parabolic restriction with non-normalized parabolic restriction.
\end{proof}

The proof of Corollary~\ref{cor_second_adj} is now more-or-less immediate; we only give a sketch as the proof closely follows the argument of \cite{HelmBC}, Theorem 11.18 (which itself is a slight modification of an argument of Bernstein).  We first note that it suffices to prove the theorem for $R = \ZZ[\frac{1}{p}]$.  In this case, one wishes to construct a natural isomorphism:
$$\Hom_{\ZZ[\frac{1}{p}] G}(I_P V, W) \cong \Hom_{\ZZ[\frac{1}{p}] M}(V, \delta_P R_{\overline{P}} W),$$
for all objects $V$ and $W$ of $\Rep_{\ZZ[\frac{1}{p}]}(M)$ and $\Rep_{\ZZ[\frac{1}{p}]}(G)$, respectively.
When $W$ is of the form $(W')^{\vee}$ for some object $W'$ of $\Rep_{\ZZ[\frac{1}{p}]}(G)$, such an isomorphism can be constructed formally from Jacquet's lemma (Lemma \ref{lem:jacquet}) using the fact that $R_P$ is left adjoint to $I_P$, together with the natural isomorphism $(I_P V)^{\vee} \cong I_P (\delta_P V^{\vee})$.

In particular the result holds when $W = \ZZ[\frac{1}{p}][G/H]_r^{\vee}$.  More generally, it holds for any direct product $\prod_{i \in I} \ZZ[\frac{1}{p}][G/H_i]_{r_i}^{\vee}$, as such a product is dual to the direct sum $\bigoplus_{i \in I} \ZZ[\frac{1}{p}][G/H_i]_{r_i}.$ Since any object of $\Rep_{\ZZ[\frac{1}{p}]}(G)$ has a resolution by direct products of copies of $\ZZ[\frac{1}{p}][G/H]_r^{\vee}$ for various $H$ and $r$ we can deduce the result in general.  

Once we have Corollary~\ref{cor_second_adj}, Corollaries~\ref{cor_noetherian} and~\ref{cor_induction_restriction} follow immediately from Corollaire 4.4 and Lemme 4.6 of~\cite{datfinitude}, respectively.

\medskip

Let us now prove Corollary~\ref{cor_integral}. Fix
$\pi\in\Irr_{\Ql}(G)$ and let $(M,\sigma)$ belong to its supercuspidal
support.  Suppose first that $\sigma$ is integral,
and let $L_{\sigma}$ be an $M$-stable $\Zl$-lattice in $\sigma$. Then for any
  parabolic  subgroup $P$ with Levi $M$, the $\Zl G$-module
  $i_{P}(L_{\sigma})$ is known to be admissible. Therefore,
  if we choose $P$ such that $\pi\hookrightarrow i_{P}(\sigma)$, the $\Zl G$-module
  $i_{P}(L_{\sigma})\cap \pi$ is a lattice in $\pi$.  Conversely, suppose
  that $\pi$ is integral and $L_{\pi}$ be a $G$-stable $\Zl$-lattice
  in $\pi$, and pick any  parabolic subgroup $P$ of $G$ with Levi
  $M$. As recalled above, we know from Iwasawa decomposition  that
  $r_{P}(L_{\pi})$ is a finitely generated $\Zl M$-module. Moreover,
   Corollary~\ref{cor_induction_restriction} (2) tells us that
  it is an admissible $\Zl M$-module. Since it is
  contained in, and generates, $r_{P}(\pi)$, this is a $\Zl M$-lattice
  in $r_{P}(\pi)$. In particular, the socle of $r_{P}(\pi)$ is
  integral. But any irreducible component of this socle is conjugate to $\sigma$ by uniqueness of
  cuspidal support.

 Here is a consequence of Corollary \ref{cor_integral} that may
 be worth mentioning.
\begin{corollary}
  Let $\pi$ be a simple integral $\Ql\G$-module such that $r_{\ell}(\pi)$ is
  cuspidal. Then $\pi$ is cuspidal. 
\end{corollary}
\begin{proof}
 Let $P\subset G$ be a parabolic subgroup and $M\subset P$ a Levi
 factor.
 As explained above, %$r_{P}(\pi)$ is integral. Therefore,
 if $L_{\pi}$ is a lattice in $\pi$, then $r_{P}(L_{\pi})$, being a finitely generated $\Zl M$-module, has to
 be a lattice in $r_{P}(\pi)$. In particular,
 $r_{\ell}(r_{P}(\pi))=r_{P}(r_{\ell}\pi)=0$, and therefore $r_{P}(\pi)=0.$
\end{proof}

\medskip

Finally, we recall the following consequence of second adjointness.
% as a consequence of  second adjointness, we can relax an hypothesis that appears
% in the literature in the proof of generic irreducibility.

\begin{corollary} %Generic irreducibility holds for any parabolic subgroup of $G$.
  Let $k$ be an algebraically closed field over $\ZM[1/p]$, let $M$ be a Levi subgroup of
  $G$,   and let $\sigma$ be an irreducible
  $kM$-module.
  \begin{enumerate}
  \item For any parabolic  subgroup $P$ of $G$ with Levi component $M$, the
    representation $i_{P}(\sigma\psi)$ is irreducible for $\psi$ in a Zariski-dense open subset of
    the $k$-torus of unramified characters of $M$.
  \item For two parabolic subgroups $P,Q$ of $G$ with Levi component $M$, we have
    $[i_{P}(\sigma)]=[i_{Q}(\sigma)]$ in the Grothendieck group of finite length $kG$-modules.
  \end{enumerate}
\end{corollary}
\begin{proof}
(1) A proof is given in Theorem 5.1 of \cite{Datnu} under the assumption that there exist
cocompact lattices in $G$. However, this hypothesis is only used in
order to get property $i)$ of Proposition 3.14
of loc.cit, and apply the implication $i)\Rightarrow iii)$ of that
proposition. But property i) of that proposition also follows from second adjointness, by
\cite[Lemme 4.12]{datfinitude}. Finally, (2) is \cite[Lemme 4.13]{datfinitude}.
\end{proof}

 \subsection*{Acknowledgements} The first author was partially
    supported by ANR grant COLOSS ANR-19-CE40-0015. The second author
    was partially supported by EPSRC New Horizons grant
    EP/V018744/1. The third author was supported by EPSRC grant
    EP/V001930/1 and the Heilbronn Institute for Mathematical
    Research. The fourth author was partially supported by NSF grant
    DMS-2001272. We would like to thank  Tony Feng for providing us with motivation and impetus for
    revisiting these finiteness problems in the representation theory of $p$-adic
    groups. We thank David Hansen and Peter Scholze for their interest and valuable
    discussions on the content of this paper. Finally we thank
    Marie-France Vign\'{e}ras and Guy Henniart for their careful reading of a preliminary version.

\bibliographystyle{alpha}
\bibliography{finiteness}

\end{document}